\documentclass[a4paper,reqno,11pt]{amsart}
\usepackage{latexsym}
\usepackage{amsmath}
\usepackage{amssymb}
\usepackage[all]{xy}
\usepackage{amsfonts}
\usepackage{verbatim}
\usepackage{amsthm}
\usepackage{mathrsfs}
\usepackage{epsfig}
\usepackage{xy}
\usepackage{array}
\usepackage{stmaryrd}
\usepackage{graphicx,color}
\usepackage{xcolor}
\usepackage{tikz}
\usetikzlibrary{arrows,calc}
\usepackage{etex}
\usepackage{mathdots}
\usepackage{float}
\usepackage{graphics}
\usepackage{pdflscape}

\usetikzlibrary{arrows,decorations.pathmorphing,decorations.pathreplacing,positioning,shapes.geometric,shapes.misc,decorations.markings,decorations.fractals,calc,patterns}

\usepackage{anysize,hyperref}
\input xypic
\xyoption{all}

\newcommand{\cal}{\mathcal}

\newtheorem{thm}{Theorem}[section]
\newtheorem{cor}[thm]{Corollary}
\newtheorem{lem}[thm]{Lemma}
\newtheorem{exm}[thm]{Example}
\newtheorem{prop}[thm]{Proposition}
\newtheorem{defn}[thm]{Definition}
\newtheorem{rem}[thm]{Remark}

\usepackage{pstricks,color}

\def\s{\stackrel}

\def\C{\mathscr{C}}
\def\D{\mathscr{D}}
\def\T{\mathcal{T}}
\def\M{\mathcal{M}}
\def\E{\mathcal{E}}

\def\S{\mathcal{S}}

\def\X{\mathscr{X}}
\def\Y{\mathscr{Y}}

\def\I{\mathcal {I}}

\def\Ker{\mbox{Ker}}

\def\Ext{\mbox{Ext}}
\def\Hom{\mbox{Hom}}

\def\Im{\mbox{Im}}

\def\la{\Lambda}

\def\sttm{\mbox{s}\tau\mbox{-tilt}\la}

\def \text{\mbox}
\def\Mod{\mathsf{Mod}}
\newcommand{\pd}{\mathsf{pd}\hspace{.01in}}
\newcommand{\add}{\mathsf{add}\hspace{.01in}}
\newcommand{\Fac}{\mathsf{Fac}\hspace{.01in}}
\newcommand{\End}{\operatorname{End}\nolimits}
\renewcommand{\Mod}{\mathsf{Mod}\hspace{.01in}}
\renewcommand{\mod}{\mathsf{mod}\hspace{.01in}}
\newcommand{\proj}{\mathsf{proj}\hspace{.01in}}
\newcommand{\co}{\mathsf{Coker}\hspace{.01in}}

\begin{document}

\title{Triangulated categories with cluster tilting subcategories}

\dedicatory{Dedicated to Professor Idun Reiten on the occasion of her 75th birthday}
\author[Yang]{Wuzhong Yang}
\address{
School of Mathematics, Northwest University, 710127, Xi'an, Shaanxi, P. R. China}
\email{yangwz@nwu.edu.cn}

\author[Zhou]{Panyue Zhou$^\ast$}
\address{
College of Mathematics, Hunan Institute of Science and Technology, 414006 Yueyang, Hunan, P. R. China.}
\email{panyuezhou@163.com}
\author[Zhu]{Bin Zhu}
\address{
Department of Mathematical Sciences,
Tsinghua University,
100084, Beijing,
P. R. China
}
\email{bzhu@math.tsinghua.edu.cn}

\begin{abstract}

Let $\C$ be a triangulated category with a cluster tilting subcategory $\T$. We introduce the notion of $\T[1]$-cluster tilting subcategories (also called ghost cluster tilting subcategories) of $\C$, which are a generalization of cluster tilting subcategories. We first develop a basic theory on ghost cluster tilting subcategories. Secondly, we study links between ghost cluster tilting theory and $\tau$-tilting theory: Inspired by the work of Iyama, J{\o}rgensen and Yang \cite{ijy}, we introduce the notion of $\tau$-tilting subcategories and tilting subcategories of $\mod\T$. We show that there exists a bijection between weak $\T[1]$-cluster tilting subcategories of $\C$ and support $\tau$-tilting subcategories of $\mod\T$. Moreover, we figure out the subcategories of $\mod\T$ which correspond to cluster tilting subcategories of $\C$. This generalizes and improves several results by Adachi-Iyama-Reiten \cite{AIR},  Beligiannis \cite{Be2}, and Yang-Zhu \cite{YZ}. Finally, we prove that the definition of ghost cluster tilting objects is equivalent to the definition of relative cluster tilting objects introduced by the first and the third author in \cite{YZ}.

\end{abstract}

\subjclass[2010]{18E30; 16G20; 16G70}

\keywords{Ghost cluster tilting subcategories; Support $\tau$-tilting subcategories; $\tau$-tilting subcategories, Cluster tilting subcategories.}

\thanks{$^\ast$Corresponding author. The first author was supported by Scientific Research Program Funded by Shaanxi Provincial Education Department (Program No. 17JK0794).
The second author is supported by the NSF of China (Grants No. 11671221) and the China Postdoctoral Science Foundation (Grant No. 2016M591155).
The third author is supported by the NSF of China (Grants No. 11671221). }

\maketitle

\tableofcontents

\bigskip

\section{Introduction}

In 2012, Adachi-Iyama-Reiten \cite{AIR} introduced the $\tau$-tilting theory for finite dimensional algebras. As a generalization of classical tilting theory, it completes tilting theory from the viewpoint of mutation. Nowadays the relationships between $\tau$-tilting theory and the various aspects of the representation theory of finite dimensional algebras have been studied. In 2013, Iyama-J{\o}rgensen-Yang \cite{ijy} gave a functor version of $\tau$-tilting theory. They considered modules over a category and showed for a triangulated category $\C$ with a silting subcategory $\S$, there exists a bijection between the set of silting subcategories of $\C$ which are in $\S\ast\S[1]$ and the set of support $\tau$-tilting pairs of mod$\S$.
\medskip

Cluster-tilted algebras, as the endomorphism algebras of cluster tilting objects in cluster categories, were introduced by Buan-Marsh-Reiten \cite{BMR}. The authors showed that the module category of a cluster-tilted algebra is equivalent to the quotient category of cluster category $\C$ by the cluster tilting object $T$, i.e. $\C/[T[1]]$ $\xrightarrow{~\sim~}  \mod\End T$. Actually, the equivalence above still holds for more general cases. It was proved that for a triangulated category $\C$ with a cluster tilting subcategory $\T$, there is an equivalence \cite{kr, KZ, IY}
$$  \C / [ \T[1] ]    \xrightarrow{~\sim~}  \mod\, \T.$$ Then we have a functor from $\C$ to $\mod\, \T$, which is denoted by $\mathbb{H}$.
Under this functor, a series of papers investigate the relationships between objects in $\C$ and modules in $\mod\, \T$ (see for example \cite{Sm, FL, hj, Be2}). Especially Adachi-Iyama-Reiten \cite{AIR} established a bijection between cluster tilting objects in a 2-Calabi-Yau triangulated category and support $\tau$-tilting modules over a cluster-tilted algebra (see also \cite{CZZ, YZZ, YZ} for various versions of this bijection). It is natural to ask which class of subcategories of $\C$ correspond bijectively
to support $\tau$-tilting subcategories of $\mod\, \T$ for higher Calabi-Yau triangulated categories or arbitrary triangulated categories.
\medskip

Motivated by this question and the bijection given by Iyama-J{\o}rgensen-Yang \cite{ijy}, we introduce the notion of $\T[1]$-cluster tilting subcategories (also called ghost cluster tilting subcategories) of $\C$, which are a generalization of cluster tilting subcategories. The first part of our work is to develop a basic theory of ghost cluster tilting subcategories of $\C$. Some intrinsic properties and results on ghost cluster tilting subcategories will be presented. Inspired also by Adachi-Iyama-Reiten \cite{AIR} and Iyama-J{\o}rgensen-Yang \cite{ijy}, we introduce the notion of $\tau$-tilting subcategories and tilting subcategories of $\mod\T$. The second part of our paper is to give some close relationships between certain ghost cluster tilting subcategories of $\C$ and some important subcategories of $\mod\T$. Here is the main result in the second part.
\medskip

\begin{thm}(see Theorem \ref{a7} and Theorem \ref{a9} for more details)
Let $\C$ be a triangulated category with a cluster-tilting subcategory $\T$.
The functor $\mathbb{H}\colon \C\to \Mod\, \T$ induces a bijection $$\Phi\colon \X\longmapsto \big(\mathbb{H}(\X),\T\cap\X[-1]\big)$$
from the first of the following sets to the second:
\begin{itemize}
\item[\emph{(I)}] $\T[1]$-rigid subcategories of $\C$.

\item[\emph{(II)}] $\tau$-rigid pairs of $\mod\, \T$.
\end{itemize}
It restricts to a bijection from the first to the second of the following sets.
\begin{itemize}
\item[\emph{(I)}] Weak $\T[1]$-cluster tilting subcategories of $\C$.

\item[\emph{(II)}] Support $\tau$-tilting subcategories of $\mod\, \T$.
\end{itemize}
\end{thm}
Consequently, we also figure out the subcategories of $\mod\, \T$ which correspond to cluster tilting subcategories of $\C$. This generalizes and improves several results in the literature.
This also answers the question above.
\medskip

When ghost cluster tilting subcategories have additive generators, we call these additive generators in $\C$ ghost cluster tilting objects. Comparing with the definition of relative cluster tilting objects in \cite{YZ}, the last part of this paper aims to show that ghost cluster tilting objects are exactly relative cluster tilting objects introduced in \cite{YZ} provided that $\C$ has a Serre functor.
\medskip

The paper is organized as follows. In Section 2, we recall some elementary definitions and facts
that will be used frequently, including cluster tilting subcategories and support $\tau$-tilting subcategories. In Section
3, we will develop a basic theory of ghost cluster tilting subcategories of $\C$. In Section 4, we explore the connections between ghost cluster tilting theory and $\tau$-tilting theory. In the last section, we show that the definition of ghost cluster tilting objects is equivalent to relative cluster tilting objects in \cite{YZ} provided that the ambient category $\C$ has a Serre functor.
\medskip

We conclude this section with some conventions.
\medskip

 Throughout this article, $k$ is an algebraically closed field. All modules we consider in this paper are left modules. Let $\C$ be an additive category. When we say that $\D$ is a subcategory of $\C$, we always assume that $\D$ is a full subcategory which is closed under isomorphisms, direct sums and direct summands. We denote by $[\D]$ the ideal of $\C$ consisting of morphisms which factor through objects in $\D$.
 Thus we get a new category $\C/[\D]$ whose objects are objects of $\C$ and whose morphisms are elements of $\C(X, Y)/[\D](X, Y)$ for $X, Y\in\C/[\D]$. For any object $M$, we denote by $\add M$ the full subcategory of $\C$ consisting of direct summands of direct sum of finitely many copies of $M$ and simply denote $\C/[\add M]$ by $\C/[M]$. Let $\X$ and $\Y$ be subcategories of $\C$. We denote by $\X \vee \Y$ the smallest subcategory of $\C$ containing $\X$ and $\Y$.  For two morphisms $f:M\rightarrow N$ and $g:N\rightarrow L$, the composition of $f$ and $g$ is denoted by $gf:M\rightarrow L$.
\medskip

Let $X$ be an object in $\C$. A morphism $f:D_0\rightarrow X$ is called a \textit{right $\D$-approximation} of $X$ if $D_0\in \D$ and Hom$_{\C}(-, f)|_{\D}$ is surjective. If any object in $\C$ has a right $\D$-approximation, we call $\D$ \textit{contravariantly finite} in $\C$. Dually, a \textit{left $\D$-approximation} and a \textit{covariantly finite subcategory} are defined. We say that $\D$ is \textit{functorially finite} if it is both covariantly finite and contravariantly finite.  For more details, we refer to \cite{ar}.
\medskip

For any triangulated category $\C$, we assume that it is $k$-linear, Hom-finite, and satisfies the Krull-Remak-Schmidt property \cite{Ha}. In $\C$, we denote the shift functor by $[1]$ and for objects $X$ and $Y$, define
$\text{Ext}_{\C}^i(X, Y)= \mbox{Hom}_{\C}(X, Y[i]).$ For two subcategories $\X, \Y$ of $\C$, we denote by Ext$^1(\X,\Y)=0$ when Ext$^1(X,Y)=0$ for $X\in \X, Y\in \Y$. For an object $X$, $|X|$ denotes the number of
non-isomorphic indecomposable direct summands of $X$.

\section{Background and preliminary results}\label{sect:0}
In this section, we give some background material and recall some results that will be used in this paper.

\subsection{Cluster tilting subcategories and relative cluster tilting objects}
Let $\C$ be a triangulated category. An important class of subcategories of $\C$ are the cluster tilting subcategories, which have many nice properties. We recall the definition of cluster tilting subcategories from \cite{bmrrt,kr,KZ,IY}.

\begin{defn}
\begin{enumerate}
\item[\emph{(1)}] A subcategory $\T$ of $\C$ is called {\rm rigid} if ${\rm Ext}_{\C}^1(\T, \T)=0$.
\item[\emph{(2)}] A subcategory $\T$ of $\C$ is called {\rm maximal rigid} if it is rigid and maximal with respect to the property: $\T=\{M\in\C\ |\ {\rm Ext}_{\C}^1(\T\vee \add M, \T\vee \add M)=0\}$.
\item[\emph{(3)}] A functorially finite subcategory $\T$ of $\C$ is called {\rm cluster tilting} if
$$\T=\{M\in\C\ |\ {\rm Ext}_{\C}^1(\T, M)=0\}=\{M\in\C\ |\ {\rm Ext}_{\C}^1(M, \T)=0\}.$$
\item[\emph{(4)}] An object $T$ in $\C$ is called {\rm cluster tilting} if $\add T$ is a cluster tilting subcategory of $\C$.
\end{enumerate}
\end{defn}
\begin{rem}\label{rem:ctsubcat}
In fact, Koenig and Zhu \emph{\cite{KZ}} indicate that a subcategory $\T$ of $\C$ is cluster tilting if and only if it is contravariantly finite in $\C$ and\ \  $\T=\{M\in\C\ |\ {\rm Ext}_{\C}^1(\T, M)=0\}$.
\end{rem}
For two subcategories $\X$ and $\Y$ of $\C$, we denote by $\X\ast \Y$ the collection of objects in $\C$ consisting of all such $M\in \C$ with triangles $$X\longrightarrow M \longrightarrow Y \longrightarrow X_0[1],$$
where $X\in \X$ and $Y\in \Y$.
\medskip

 Recall from \cite{BK} that $\C$ has a Serre functor $\mathbb{S}$ provided $\mathbb{S}:\C\rightarrow \C$ is an equivalence and there exists a functorial isomorphism
$${\rm Hom}_{\C}(A, B)\simeq D{\rm Hom}_{\C}(B, \mathbb{S}A)$$
for any $A, B\in \C$, where $D$ is the duality over $k$. Thus $\C$ has the  Auslander-Reiten translation
$\tau_{\C}\simeq\mathbb{S}[-1]$ \cite{RVdB}. Define an equivalence $F=\tau_{\C}^{-1}\circ [1]$. We call an object $M$ in $\C$ \emph{$F$-stable} if $F(M)\simeq M$ and a subcategory $\M$ of $\C$ \emph{$F$-stable} if $F(\M)=\M$. We say that $\C$ is \emph{$2$-Calabi-Yau}
if $\mathbb{S}\simeq [2]$. For a $2$-Calabi-Yau category $\C$, $F=id_{\C}$.
\medskip

We have the following result \cite{kr, KZ, IY}, which will be used frequently in this paper.
\begin{prop}\label{FT}
Let $\T$ be a cluster-tilting subcategory of $\C$. Then
\begin{enumerate}
\item[\emph{(a)}] $\C=\T\ast \T[1]$.
\item[\emph{(b)}] $F\T=\T$ if $\C$ has a Serre functor.
\end{enumerate}
\end{prop}

If there is a cluster tilting object $T$ in $\C$, then any cluster tilting subcategory $\T'$ is of the form $\add\, T'$ for some cluster tilting object $T'$, and the numbers of the non-isomorphic indecomposable direct summands of $T,$ and $T'$ are the same \cite{YZZ}. We denote this number by $r(\C)$, which is called the cluster rank of the triangulated category $\C$. 
When $\C$ is 2-Calabi-Yau, we can define the mutation of cluster tilting objects. In order to generalise it in a more general triangulated category, Yang and Zhu \cite{YZ} introduced the notion of relative cluster-tilting objects
 in triangulated categories.
\begin{defn}
Let $\C$ be a triangulated category with a Serre functor and with cluster tilting objects.
\begin{itemize}
\item An object $X$ in $\C$ is called {\rm ghost rigid} if there exists a cluster tilting object $T$ such that $[T[1]](X, X[1])=0$. In this case, $X$ is also called $T[1]$-{\rm rigid}.    
\item An object $X$ in $\C$ is called {\rm relative cluster tilting} if there exists a cluster tilting object $T$ such that $X$ is $T[1]$-rigid and $|X|=r(\C)$. In this case, $X$ is also called $T[1]$-{\rm cluster tilting}.
\end{itemize}
\end{defn}
Throughout this paper, we denote by $T[1]\mbox{-}{\rm rigid}\,\C$ (respectively, $T[1]\mbox{-}{\rm cluster\ tilt}\,\C$) the set of isomorphism classes of basic $T[1]$-rigid (respectively, basic $T[1]$-cluster tilting) objects in $\C$.

\subsection{Support $\tau$-tilting modules and support $\tau$-tilting subcategories}
Let $\la$ be a finite dimensional $k$-algebra and $\tau$ the Auslander-Reiten translation. We denote by $\proj\la$ the subcategory of $\mod\la$ consisting of projective $\la$-modules. Support $\tau$-tilting modules were introduced by Adachi, Iyama and Reiten \cite{AIR}, which can be regarded as a generalization of tilting modules.
\begin{defn}Let $(X, P)$ be a pair with $X\in \mod\la$ and $P\in \proj\la$.
\begin{enumerate}
\item[\emph{1.}] $X$ is called $\tau${\rm -rigid} if {\rm Hom}$_{\la}(X,\tau X)=0$.
\item[\emph{2.}] $X$ is called $\tau${\rm -tilting} if $X$ is $\tau$-rigid and $|X|=|\la|$.
 \item[\emph{3.}] $(X, P)$ is called a $\tau${\rm -rigid pair} if $X$ is $\tau$-rigid and {\rm Hom}$_{\la}(P, X)=0$. 
\item[\emph{4.}] $(X, P)$ is said to be a {\rm support $\tau$-tilting pair} if it is a $\tau$-rigid pair and $|X|+|P|=|\la|$. In this case, $X$ is called a {\rm support $\tau$-tilting module}.
\end{enumerate}
\end{defn}
Throughout this paper, we denote by $\tau$-rigid$\la$ the set of isomorphism classes of basic $\tau$-rigid pairs of $\la$, and by $\sttm$ the set of isomorphism classes of basic support $\tau$-tilting $\la$-modules.
\medskip

The following proposition gives a criterion for a $\tau$-rigid $\la$-module to be a support $\tau$-tilting $\la$-module.
\begin{prop}
  \label{y4}{\emph{\cite[Proposition 2.14]{ja}}}
Let $\la$ be a finite dimensional algebra and $M$ a $\tau$-rigid $\la$-module.
  Then $M$ is a support $\tau$-tilting $\la$-module if and only if
  there exists an exact sequence
  $$
    \la \xrightarrow{f} M' \xrightarrow{g} M''\to 0,
$$
  with $M',M''\in \add M$ and $f$ a left $(\add M)$-approximation
  of $\la$.
\end{prop}

In 2013, Iyama, J{\o}rgensen and Yang \cite[Definition 1.3]{ijy} extended the notion of support $\tau$-tilting modules for finite dimensional algebras over fields to that for essentially small additive categories. Let $\T$ be an additive category.
We write $\Mod\, \T$ for the abelian category of contravariant additive functors from $\T$ to the category of abelian groups and
$\mod\, \T$ for the full subcategory of finitely presented functors, see \cite{au1}.
\medskip


\begin{defn}{\emph{\cite[Definition 1.3]{ijy}}}\label{a1}
Let $\T$ be an essentially small additive category.
\begin{itemize}
\item[\emph{(i)}] Let $\M$ be a subcategory of $\mod\,\T$.  A class $\{\,
  P_1 \stackrel{ \pi^M }{ \rightarrow } P_0 \rightarrow M \rightarrow
  0 \,\mid\, M \in \M \,\}$ of projective presentations in $\mod\,
  \T$ is said to have \emph{Property (S)} if
\[
  \emph{\Hom}_{ \mod\,\T }( \pi^M , M' )
  : \emph{\Hom}_{ \mod\,\T }( P_0 , M' )
    \rightarrow \emph{\Hom}_{ \mod\, \T }( P_1 , M' )
\]
is surjective for any $M , M' \in \M$.

\item[\emph{(ii)}] A subcategory $\M$ of $\mod\, \T$ is said to be
  \emph{$\tau$-rigid} if there is a class of projective presentations
  $\{P_1 \rightarrow P_0 \rightarrow M \rightarrow 0\mid M\in\M\}$
  which has Property (S).

\item[\emph{(iii)}] A \emph{$\tau$-rigid pair} of $\mod\, \T$ is a pair $(
  \M , \E )$, where $\M$ is a $\tau$-rigid subcategory of $\mod\,
  \T$ and $\E \subseteq \T$ is a subcategory with $\M( \E )=0$,
  that is, $M( E )=0$ for each $M \in \M$ and $E \in \E$.

\item[\emph{(iv)}] A $\tau$-rigid pair $( \M , \E )$ is \emph{support
  $\tau$-tilting} if $\E = \emph{Ker}\,( \M )$ and for each $T\in \T$
  there exists an exact sequence $\T( - , T )
  \stackrel{f}{\rightarrow} M^0 \rightarrow M^1 \rightarrow 0$ with
  $M^0, M^1 \in \M$ such that $f$ is a left $\M$-approximation. In this case, $\M$ is called a \emph{support
  $\tau$-tilting subcategory} of $\mod\T$.
\end{itemize}
\end{defn}

\subsection{From triangulated categories to abelian categories}
In this subsection, we assume that $\cal T$ is a cluster tilting subcategory of a triangulated category $\C$. A $\T$-module is a contravariant $k$-linear functor $F:\T\rightarrow \Mod k$. Then $\T$-modules form an abelian category $\Mod\T$. We denote by $\mod\T$ the subcategory of $\Mod\T$ consisting of finitely presented $\T$-modules. It is easy to know that $\mod\T$ is an abelian category. Moreover the restricted Yoneda functor
$$\mathbb{H}\colon\C \rightarrow \Mod\, \T, \;\;
  M \mapsto \Hom_{\C}( -,M )\mid_{\T}$$
  is homological and induces an equivalence
  $$
  \T
    \xrightarrow{~\sim~}  \proj(\mod\, \T).
$$

The following results are crucial in this paper.

\begin{thm}\label{a5}
\begin{itemize}
\item[\emph{(i)}] $\mathbb{H}(\C)$ is a subcategory of $\mod\T$.

\item[\emph{(ii)}] \emph{\cite{au1}} For
    $N \in \Mod\, \T$ and $T \in \T$, there exists a
    natural isomorphism
\[
 \emph{\Hom}_{ \Mod\, \T }\big( \T( -,T ) , N \big)
  \xrightarrow{~\sim~} N( T ).
\]

  \item[\emph{(iii)}] \emph{\cite{kr,KZ,IY}} The functor \ $\mathbb{H}$\  \  from \emph{(i)}
    induces an equivalence
$$
  \C / [ \T[1] ]
    \xrightarrow{~\sim~}  \mod\, \T,
$$ and $\mod\, \T$ is Gorenstein of dimension at most one.
\end{itemize}
\end{thm}
\proof Since $\T$ is cluster tilting, for any object $C\in\C$, there exists a triangle
$$\xymatrix{T_0\ar[r]^{f}&T_1\ar[r]^{g}&C\ar[r]^{h}&T_0[1]},$$where $T_0,T_1\in\T$. Applying the functor $\mathbb{H}$ to the above triangle, we get an exact sequence
$$\xymatrix{\mathbb{H}(T_0)\ar[r]^{f\circ}&\mathbb{H}(T_1)\ar[r]^{g\circ}&\mathbb{H}(C)\ar[r]& 0}.$$
This shows that $\mathbb{H}(C)\in\mod\T$.
 \qed
\medskip

If there exists an object $T\in \C$ such that $\T=\add T$, we obtain the following.
\begin{cor}\label{y2}
Let $T$ be a cluster tilting object in $\C$ and $\la=\End^{\emph{op}}_{\C}(T)$. Then the functor
\begin{equation}\label{w0}
\begin{array}{l}
\overline{(-)}:= {\rm Hom}_{\C}(T,-):\C\longrightarrow \mod\la\end{array}
\end{equation}
 induces an equivalence
\begin{equation}\label{w1}
\begin{array}{l}
\C/[T[1]]\s{\sim}\longrightarrow \mod\la.
\end{array}
\end{equation}
\end{cor}

This equivalence gives a close relationship between the relative cluster tilting objects in $\C$ and support $\tau$-tilting $\la$-modules.

\begin{thm}\label{y3}{\emph{\cite[Theorem 3.6]{YZ}}}
Let $\C$ be a triangulated category with a Serre functor $\mathbb{S}$ and a cluster tilting object $T$, and let $\la={\rm End}^{\emph{op}}_{\C}(T)$. Then the functor (\ref{w0}) induces the following bijections
\begin{eqnarray*}
T[1]\mbox{-}{\rm rigid}\,\C & \s{(a)}\longleftrightarrow & \tau\mbox{-}{\rm rigid}\la,\\
T[1]\mbox{-}{\rm cluster\, tilt}\,\C & \s{(b)}\longleftrightarrow & {\rm s}\tau\mbox{-}{\rm tilt}\la.\\
\end{eqnarray*}
\end{thm}

\section{Ghost cluster tilting subcategories}\label{sect:1}
\setcounter{equation}{0}

In this section, our aim is to develop a basic theory of ghost cluster tilting subcategories in a triangulated category with cluster tilting subcategories. We first give the definitions of related subcategories and then discuss connections between them.
\begin{defn}\label{a2}
Let $\C$ be a triangulated category with a cluster tilting subcategory.
\begin{itemize}
\item[\emph{(i)}] A subcategory $\X$ in $\C$ is called  \emph{ghost rigid }if there exists a cluster tilting subcategory $\T$ such that $[\T[1]](\X, \X[1])=0$. In this case, $\X$ is also called $\T[1]$\emph{-rigid}.

\item[\emph{(ii)}] A subcategory $\X$ in $\C$ is called \emph{maximal ghost rigid} if there exists a cluster tilting subcategory $\T$ such that $\X$ is $\T[1]$-rigid and
$$[\T[1]](\X\vee\add M, (\X\vee\add M)[1])=0 \text{ implies } M\in\X. $$
In this case, $\X$ is also called \emph{maximal} $\T[1]$\emph{-rigid}.

 \item[\emph{(iii)}] A subcategory $\X$ in $\C$ is called {\rm weak ghost cluster tilting} if there exists a cluster tilting subcategory $\cal T$ such that $\T\subseteq \X[-1]\ast\X$ and $$\X=\{M\in\C\ |\ [\T[1]](M, \X[1])=0 \mbox{ and } [\T[1]](\X, M[1])=0 \ \}.$$
     In this case, $\X$ is also called \emph{weak} $\T[1]$\emph{-cluster tilting}.

\item[\emph{(iv)}] A subcategory $\X$ in $\C$ is called {\rm ghost cluster tilting} if $\X$ is contravariantly finite in $\C$ and there exists a cluster tilting subcategory $\cal T$ such that $$\X=\{M\in\C\ |\ [\T[1]](M, \X[1])=0\mbox{ and } [\T[1]](\X, M[1])=0 \ \}.$$
     In this case, $\X$ is also called $\T[1]$\emph{-cluster tilting}.

\item[\emph{(v)}] An object $X$ is called $\T[1]$-rigid, maximal $\T[1]$-rigid, weak $\T[1]$-cluster tilting, or $\T[1]$-cluster tilting if $\add X$ is  $\T[1]$-rigid, maximal $\T[1]$-rigid, weak $\T[1]$-cluster tilting, or $\T[1]$-cluster tilting respectively.
\end{itemize}
\end{defn}

\medskip

\begin{rem} We will show in the last section that ghost cluster tilting objects are exactly the relative cluster tilting objects defined by the first and the third author in\cite{YZ}.\end{rem}
\medskip

From now on to the end of the section, we prove some properties of ghost cluster tilting subcategories. We first give an observation.

\begin{prop}\label{prop:ctsubcat}
 Cluster-tilting subcategories of $\C$ are ghost cluster tilting. More precisely, cluster tilting subcategories in $\C$ are $\T[1]$-cluster tilting, where $\T$ is a  cluster tilting subcategory.
\end{prop}
\begin{proof}
Let $\X$ be an arbitrary cluster tilting subcategory in $\C$. Clearly, $\X$ is contravariantly finite and $$\X\subseteq \{M\in\C\ |\ [\T[1]](\X, M[1])=0= [\T[1]](M, \X[1])\}.$$
For any object $M\in \{M\in\C\ |\ [\T[1]](\X, M[1])=0= [\T[1]](M, \X[1])\}$, since $\T$ is cluster tilting,
there exists a triangle
$$\xymatrix{T_1\ar[r]^{f}&T_0\ar[r]^{g}&M\ar[r]^{h\;\,}&T_1[1]},$$
where $T_0,T_1\in\T$. Take a left $\X$-approximation of $T_0$ and complete it to a triangle
$$\xymatrix{T_0\ar[r]^{u}&X_1\ar[r]^{v}&X_2\ar[r]^{w\;\,}&T_0[1]},$$
where $X_1\in\X$. Since $\X$ is cluster tilting, we know that $X_2\in\X$.
By the octahedral axiom, we have a commutative diagram
$$\xymatrix{
T_1\ar[r]^{f}\ar@{=}[d]&T_0\ar[r]^{g}\ar[d]^{u}&M\ar[r]^{h}\ar[d]^{a}&T_1[1]\ar@{=}[d]\\
T_1\ar[r]^{x=uf}&X_1\ar[r]^{y}\ar[d]^{v}&N\ar[r]^{z}\ar[d]^{b}&T_1[1]\\
&X_2\ar@{=}[r]\ar[d]^{w}&X_2\ar[d]^{c}\\
&T_0[1]\ar[r]^{g[1]}&M[1]}$$
of triangles. We claim that $x$ is a left $\X$-approximation of $T_1$.
Indeed, for any
morphism $\alpha\colon T_1\to X'$, where $X'\in\X$, since $\alpha\circ h[-1]\in[\T](M[-1],\X)=0$, there exists a morphism $\beta\colon T_0\to X'$ such that $\alpha=\beta f$. Since $u$ is a left $\X$-approximation of $T_0$ and $X'\in\X$,  there exists a morphism $\gamma\colon X_1\to X'$ such that
$\beta=\gamma u$ and then $\alpha=\gamma(uf)=\gamma x$. This shows that $x$ is a left $\X$-approximation of $T_1$.
Note that $\X$ is cluster tilting. Thus $N\in\X$. Since $c=g[1]w\in [\T[1]](\X,M[1])=0$.
This shows that the triangle
$$\xymatrix{M\ar[r]^{a}&N\ar[r]^{b}&X_2\ar[r]^{c\;\;}&M[1].}$$
splits.  It follows that $M$ is a direct summand of $N$ and then $M\in\X$. Thus
$$\X= \{M\in\C\ |\ [\T[1]](\X, M[1])=0= [\T[1]](M, \X[1])\}$$ and hence $\X$ is $\T[1]$-cluster tilting.
\end{proof}

The following example shows that ghost cluster tilting subcategories need not be cluster tilting.

\begin{exm}
{\upshape Let $A=kQ/I$ be a self-injective algebra given by the quiver $$Q:
\setlength{\unitlength}{0.03in}\xymatrix{1 \ar@<0.5ex>[r]^{\alpha}_{\ } & 2\ar@<0.5ex>[l]^{\beta}}$$ and $I=\langle\alpha\beta\alpha\beta, \beta\alpha\beta\alpha\rangle$. Let $\C$ be the stable module category
${\rm \underline{\mod}}A$ of $A$. This is a triangulated category whose Auslander-Reiten quiver is the following:}
\end{exm}
\vspace{-2ex}
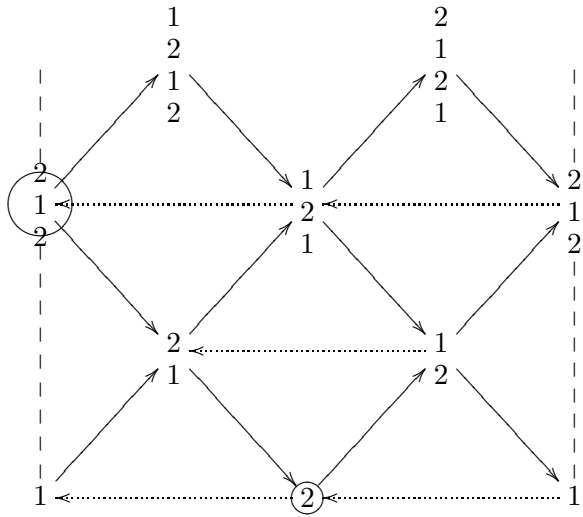
\begin{center}
$
\xymatrix@!@C=0.01cm@R=0.2cm{
&\txt{1\\2\\1\\2}\ar[dr]&&\txt{2\\1\\2\\1}\ar[dr]&\\
*+[o][F-]{\txt{2\\1\\2}}\ar[ur]\ar[dr]\ar@{--}[u]&&\txt{1\\2\\1}\ar@{.>}[ll]\ar[ur]\ar[dr]&&\txt{2\\1\\2}\ar@{.>}[ll]\ar@{--}[u]\\
&\txt{2\\1}\ar[dr]\ar[ur]&&\txt{1\\2}\ar[dr]\ar[ur]\ar@{.>}[ll]&\\
 1\ar[ur]\ar@{--}[uu]&&*+[o][F-]{2}\ar@{.>}[ll]\ar[ur]&&1\ar@{.>}[ll]\ar@{--}[uu] }
$
\end{center}
where the leftmost and rightmost columns are identified.{\setlength{\jot}{-3pt}
It is easy to see that $\T:=\add(2\oplus \begin{aligned}2\\1\\2\end{aligned})$ is a cluster tilting subcategory of $\C$.
Note that $\X:=\add(2\oplus \begin{aligned}1\\2\end{aligned})$ is $\T[1]$-tilting of $\C$, but not cluster tilting, since $\Hom(\begin{aligned}1\\2\end{aligned},\begin{aligned}1\\2\end{aligned}[1])=\Hom(\begin{aligned}1\\2\end{aligned},\begin{aligned}1\\2\end{aligned})\neq0$.
}

\medskip

In order to give a precise relation between cluster tilting and ghost cluster tilting subcategories, we assume that $\C$ has a Serre functor $\mathbb{S} $ \cite{BK}. We recall that $F=\tau_{\C}^{-1}[1]=\mathbb{S}^{-1}[2]$. The following result gives a characterisation of cluster tilting subcategories in term of ghost cluster tilting subcategories, which implies that in $2$-Calabi-Yau triangulated categories, ghost cluster tilting subcategories coincide with cluster tilting subcategories.

\begin{thm}\label{thm:F-stable}
Let $\C$ be a triangulated category with a Serre functor $\mathbb{S}$ and a cluster tilting subcategory. Then $F$-stable ghost cluster tilting subcategories of $\C$ are precisely cluster tilting subcategories.
\end{thm}
To prove this theorem, we need the following easy observation.
\begin{lem}\label{lem:T[1]-rigid}
For two objects $M$ and $N$ in $\C$, $[\T[1]](M, N[1])=0$ and $[\T[1]](N, \tau_\C M)=0$ if and only if ${\rm Hom}_{\C}(M, N[1])=0$. In particular, if $\C$ is 2-CY, then $M$ is $\T[1]$-rigid if and only if $M$ is rigid.
\end{lem}
\begin{proof}
 Using similar arguments as in the proof of Proposition 3.4 in \cite{YZ}.
\end{proof}
Now we are ready to prove Theorem \ref{thm:F-stable}.
\medskip

It follows from Theorem \ref{FT} and Proposition \ref{prop:ctsubcat} that cluster tilting subcategories are $F$-stable ghost cluster tilting. Now prove the other direction. Let $\X$ be a $\T[1]$-cluster tilting subcategory satisfying $F\X=\X$, where $\T$ is a cluster tilting subcategory. It follows that $\tau_{\C}\X=\X[1]$.
\medskip

(1) We show that $\X$ is a rigid subcategory of $\C$. For any two objects $M, N\in \X$, since $\X$ is $\T[1]$-tilting, we have
\begin{equation}\label{F-stabl1}
\begin{array}{l}
[\T[1]](M, N[1])=0.
\end{array}
\end{equation}
Similarly, since $\tau_{\C}\X=\X[1]$, we have that $\tau_{\C}M=M'[1]$, where $M'\in \X$. It follows that
   \begin{equation}\label{F-stabl2}
\begin{array}{l}
[\T[1]](N, \tau_\C M)  = [\T[1]](N, M'[1])=0.
\end{array}
\end{equation}
By Lemma \ref{lem:T[1]-rigid}, equalities (\ref{F-stabl1}) and (\ref{F-stabl2}) imply that $\Hom_\C(M, N[1])=0$.
\medskip

(2) We show that $\X=\{M\in\C\ |\ {\rm Ext}_{\C}^1(\X, M)=0\}$. The `$\subseteq$' part is clear. Assume that an object $M\in\C$ satisfies ${\rm Ext}_{\C}^1(\X, M)=0$. Then
\begin{eqnarray*}
\Hom_\C(M,\X[1]) & \simeq & D\Hom_\C(\X[1], \mathbb{S}M)  \nonumber \\
                 & \simeq & D\Hom_\C(\tau\X, F\mathbb{S}M) \nonumber \\
                 & \simeq & D\Hom_\C(\X, M[1])=0.
\end{eqnarray*}
This implies that $[\T[1]](M,\X[1])=0$. Since $[\T[1]](\X, M[1])=0$ and $\X$ is $\T[1]$-cluster tilting, we know that $M\in\X$.
\medskip

Note that $\X$ is contravariantly finite. It follows from Remark \ref{rem:ctsubcat} that $\X$ is a cluster tilting subcategory of $\C$.   \qed
\medskip

From Definition \ref{a2}, any ghost cluster tilting subcategory is a contravariantly finite maxiaml ghost rigid subcategory. Now we prove the converse holds, which generalizes a result \cite[Theorem 2.6]{zz}.

\begin{thm}\label{a3}
Let $\C$ be a triangulated category with a cluster tilting subcategory $\T$. Then
ghost cluster tilting subcategories are precisely contravariantly finite maximal ghost rigid subcategories.
\end{thm}

To prove Theorem \ref{a3}, we shall need the following result.

\begin{lem}\label{lem:main}
\begin{itemize}
\item[\emph{(a)}] Let $\X$ be a maximal $\T[1]$-rigid subcategory in $\C$. For $T_0\in\T$, if there is a triangle: $$M[-1]\s{f}\longrightarrow T_0 \s{g}\longrightarrow X_0 \s{h}\longrightarrow M$$ such that $X_0\in\X$ and $g\colon T_0\longrightarrow X_0$ is a left $\X$-approximation of $T_0$, then $M\in \X$.
\item[\emph{(b)}] Let $\X$ be a maximal $\T[1]$-rigid subcategory in $\C$. For $T_0\in\T$, if there is a triangle: $$X_0[-1]\s{f}\longrightarrow T_0 \s{g}\longrightarrow M\s{h}\longrightarrow X_0$$ such that $X_0\in\X$ and $f\colon X_0[-1]\longrightarrow T_0$ is a right $\X[-1]$-approximation of $T_0$, then $M\in \X$.
    \end{itemize}
\end{lem}

\proof We only prove part (a), the proof of part (b) is similar. For any $ x\in [\T](M[-1], \X),$ there are two morphisms $x_1\colon M[-1]\rightarrow T_1$ and $x_2\colon T_1\rightarrow X_1$ such that $x=x_2x_1$, where $T_1\in \T$ and $X_1\in\X$.

\begin{center}
$\xymatrix@!@C=0.6cm@R=0.4cm{
X_0[-1]\ar[r]^{h[-1]} & M[-1]\ar[r]^{f}\ar[d]^{x_1} & T_0\ar[r]^{g}\ar@{.>}[ddl]^{a} & X_0\ar[r]^{h}\ar@{.>}[ddll]^{b} & M\ar[r]^{-f[1]\quad} & T_0[1] \\
                       & T_1\ar[d]^{x_2}            &               &               & T_2\ar[u]^{y_2}\ar@{.>}[lu]^{c}    \\
                      & X_1                         &                &               & X_2[-1]\ar[u]^{y_1}
  }$

\medskip
\end{center}
Since $\X$ is $\T[1]$-rigid, we have $xh[-1]=x_2(x_1h[-1])=0$. Thus, there exists $a\colon T_0\to X_1$ such that $x=af$. Because $g$ is a left $\X$-approximation of $T_0$, we know there exists $b\colon X_0\to X_1$ such that $a=bg$. Therefore, $x=af=b(gf)=0$ and
\begin{equation}\label{equi:lem1}
\begin{array}{l}
[\T[1]](M, \X[1])=0.
\end{array}
\end{equation}
  For $ y\in [\T](\X[-1], M),$ there are two morphisms $y_1\colon X_2[-1]\rightarrow T_2$ and $y_2\colon T_2\rightarrow M$ such that $y=y_2y_1$, where $T_2\in \T$ and $X_2\in\X$. Since $f[1]y_2=0$, there exists $c\colon T_2\to X_0$ such that $y_2=hc$. Because $\X$ is $\T[1]$-rigid, we have $y=y_2y_1=h(cy_1)=0.$ Therefore,
\begin{equation}\label{equi:lem2}
\begin{array}{l}
[\T[1]](\X, M[1])=0.
\end{array}
\end{equation}
$\forall z\in [\T](M[-1], M),$ there are two morphisms $z_1\colon M[-1]\rightarrow T_3$ and $z_2\colon T_3\rightarrow M$ such that $z=z_2z_1$, where $T_3\in \T$. Since $f[1]z_2=0$, there exists $d\colon T_3\to X_0$ such that $z_2=hd$. By equality (\ref{equi:lem1}), we have $z=z_2z_1=h(dz_1)=0$. Thus,
\begin{equation}\label{equi:lem3}
\begin{array}{l}
[\T[1]](M, M[1])=0.
\end{array}
\end{equation}
\begin{center}
$\xymatrix@!@C=0.5cm@R=0.3cm{
 T_0\ar[r]^{g} & X_0\ar[r]^{h} & M\ar[r]^{-f[1]\quad} & T_0[1] \\
 &               & T_3\ar[u]^{z_2}\ar@{.>}[lu]^d    \\
&               & M[-1]\ar[u]^{z_1}
  }$
\end{center}
Using equalities (\ref{equi:lem1}), (\ref{equi:lem2}) and (\ref{equi:lem3}), we get $[\T[1]](\X\vee\add M, (\X\vee\add M)[1])=0$. Note that $\X$ is maximal $\T[1]$-rigid. Hence $M\in \X$.

 \qed
\medskip

We have the following direct consequences.

\begin{cor}\label{cor:subset}  Let $\T$ be a cluster tilting subcategory in a triangulated category $\C$ and $\X$ be a covariantly (or contravariantly) finite maximal $\T[1]$-rigid subcategory. Then $$\T\subseteq \X[-1]* \X.$$
\end{cor}

\begin{cor} Let $\C$ be a triangulated category with  a cluster tilting subcategory $\T$. Then
ghost cluster tilting subcategories are weak ghost cluster tilting subcategories.
\end{cor}

\proof This follows from Theorem \ref{a3} and Corollary \ref{cor:subset}.  \qed
\medskip

The following example shows that the converse is not true.

\begin{exm}
Let $\C$ be a cluster category of type $\mathbb{A}_\infty$ in \emph{\cite{hj, Ng}}. The AR-quiver of $\C$ is as follows:
$$
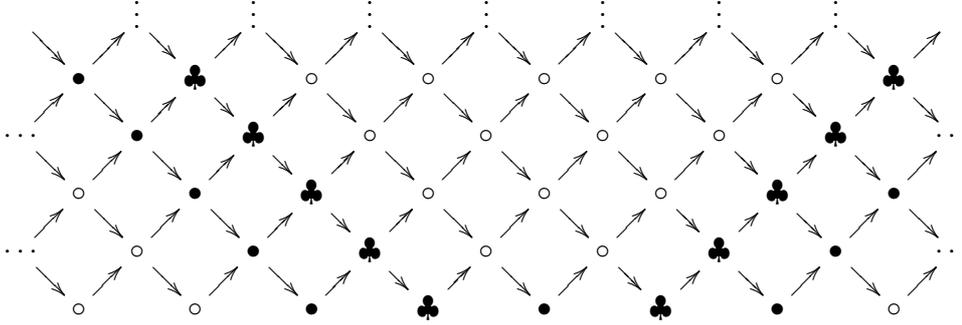
\xymatrix @-2.1pc @! {
    \rule{2ex}{0ex} \ar[dr] & & \vdots \ar[dr]  & & \vdots \ar[dr]& & \vdots \ar[dr]& & \vdots \ar[dr] & & \vdots \ar[dr] & & \vdots \ar[dr] & & \vdots \ar[dr] & & \rule{2ex}{0ex} \\
    & \bullet \ar[ur] \ar[dr] & & \clubsuit \ar[ur] \ar[dr] & &\circ \ar[ur] \ar[dr]  & & \circ \ar[ur] \ar[dr]& & \circ \ar[ur] \ar[dr] & & \circ \ar[ur] \ar[dr] & & \circ \ar[ur] \ar[dr] & & \clubsuit \ar[ur] \ar[dr] & \\
    \cdots \ar[ur]\ar[dr]& & \bullet \ar[ur] \ar[dr] & & \clubsuit \ar[ur] \ar[dr] &  & \circ \ar[ur] \ar[dr] && \circ \ar[ur] \ar[dr] & & \circ \ar[ur] \ar[dr] & & \circ \ar[ur] \ar[dr] & & \clubsuit \ar[ur] \ar[dr] & & \cdots \\
    & \circ \ar[ur] \ar[dr] & & \bullet \ar[ur] \ar[dr] & &\clubsuit \ar[ur] \ar[dr] & &\circ \ar[ur] \ar[dr] & & \circ \ar[ur] \ar[dr] & & \circ \ar[ur] \ar[dr] & & \clubsuit\ar[ur] \ar[dr] & & \bullet \ar[ur] \ar[dr] & \\
    \cdots \ar[ur]\ar[dr]& & \circ \ar[ur] \ar[dr] &  & \bullet \ar[ur] \ar[dr] && \clubsuit \ar[ur] \ar[dr]& & \circ \ar[ur] \ar[dr] & & \circ \ar[ur] \ar[dr] & & \clubsuit \ar[ur] \ar[dr] & & \bullet\ar[ur] \ar[dr] & & \cdots\\
    & \circ \ar[ur] & & \circ \ar[ur] & &\bullet \ar[ur] & & \clubsuit \ar[ur] & & \bullet \ar[ur] & & \clubsuit \ar[ur] & & \bullet \ar[ur] & & \circ \ar[ur] & \\ }$$
Set $\X$ to be the subcategory whose indecomposable objects are marked by bullets here,  $\T$ the subcategory whose indecomposable objects are marked by clubsuits here. It is easy to see that $\T$ is a cluster tilting subcategory of $\C$ and $\T\subseteq\X[-1]\ast\X$.
By Theorem 4.3 and Theorem 4.4 in \emph{\cite{hj}}, we know that $\X$ is a weak cluster tilting subcategory of $\C$
and $\X$ is not contravariantly finite in $\C$. That is to say, $\X$ is a weak ghost cluster tilting subcategory in the sense of Definition \ref{a2}, but it is not ghost cluster tilting ($=$cluster tilting).
\end{exm}

Now we are ready to prove Theorem \ref{a3}.
\medskip

(1)\ \  Assume that $\X$ is a $\T[1]$-cluster tilting subcategory. If there exists an object $M\in \C$ such that $$[\T[1]](\X\vee\add M, (\X\vee\add M)[1])=0,$$ then $$[\T[1]](\X, M[1])=0\  \text{ and }\  [\T[1]](M, \X[1])=0.$$ Since $\X$ is $\T[1]$-cluster tilting, we obtain $M\in\X.$
\medskip

(2)\ \ Assume that $\X$ is a contravariantly finite maximal $\T[1]$-rigid subcategory in $\C$.
Clearly, $$\X\subseteq \{M\in\C\ |\ [\T[1]](\X, M[1])=0= [\T[1]](M, \X[1])\}.$$
For any object $M\in \{M\in\C\ |\ [\T[1]](\X, M[1])=0= [\T[1]](M, \X[1])\}$, since $\T$ is cluster tilting,
there exists a triangle
$$\xymatrix{T_1\ar[r]^{f}&T_0\ar[r]^{g}&M\ar[r]^{h\;\;}&T_1[1]},$$
where $T_0,T_1\in\T$. By Corollary \ref{cor:subset}, there exists a triangle
$$\xymatrix{T_0\ar[r]^{u}&X_1\ar[r]^{v}&X_2\ar[r]^{w\;\;}&T_0[1]},$$
where $X_1,X_2\in\X$. Since $\X$ is $\T[1]$-rigid, we have that $u$ is a left $\X$-approximation of $T_0$.
By the octahedral axiom, we have a commutative diagram
$$\xymatrix{
T_1\ar[r]^{f}\ar@{=}[d]&T_0\ar[r]^{g}\ar[d]^{u}&M\ar[r]^{h}\ar[d]^{a}&T_1[1]\ar@{=}[d]\\
T_1\ar[r]^{x=uf}&X_1\ar[r]^{y}\ar[d]^{v}&N\ar[r]^{z}\ar[d]^{b}&T_1[1]\\
&X_2\ar@{=}[r]\ar[d]^{w}&X_2\ar[d]^{c}\\
&T_0[1]\ar[r]^{g[1]}&M[1]}$$
of triangles. Using similar arguments as in the proof of Proposition \ref{prop:ctsubcat},  we conclude that
$x$ is a left $\X$-approximation of $T_1$.
By Lemma \ref{lem:main}, we have $N\in\X$. Since $$c=g[1]w\in [\T[1]](\X,M[1])=0.$$
This shows that the triangle
$$\xymatrix{M\ar[r]^{a}&N\ar[r]^{b}&X_2\ar[r]^{c\;\;}&M[1].}$$
splits. It follows that $M$ is a direct summand of $N$ and then $M\in\X$. Hence $\X$ is $\T[1]$-cluster tilting.
 \qed
\medskip

As an application of the above theorem, we have the following.

\begin{cor}\label{a4}{\emph{\cite[Theorem 2.6]{zz}}}
Let $\C$ be a $2$-Calabi-Yau triangulated category with a cluster tilting subcategory $\cal T$. Then every
functorially finite maximal rigid subcategory is cluster-tilting.
\end{cor}
\proof This follows from Lemma \ref{lem:T[1]-rigid} and Theorem \ref{a3}.
 \qed
\medskip

We give a characterization of weak ghost cluster tilting subcategories.

\begin{thm}
Let $\C$ be a triangulated category with a cluster tilting subcategory $\T$, and $\X$ a subcategory of $\C$. Then
$\X$ is a weak ghost cluster tilting subcategory if and only if $\X$ is a maximal ghost rigid subcategory such that $\T\subseteq\X[-1]\ast\X$.
\end{thm}

\begin{proof}
 Using similar arguments as in the proof of Theorem \ref{a3}.
\end{proof}

We conclude this section with a picture illustrating the relationships between ghost cluster tilting subcategories and related subcategories:
$$
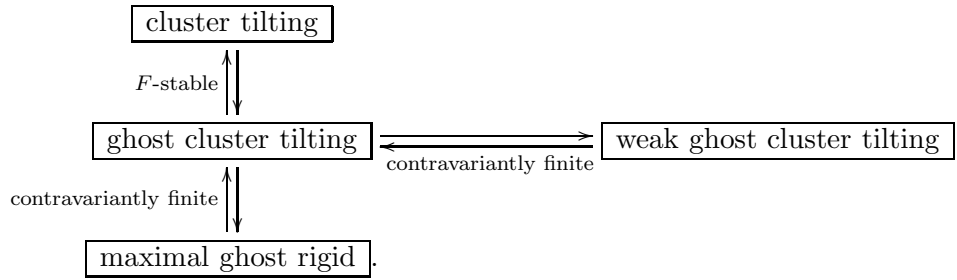
\xymatrix@C=2.7cm{
{\begin{tabular}{|c|}
\hline
\mbox{cluster tilting}\ar@<+2pt>[d]\\
\hline
\end{tabular}}\\
{\begin{tabular}{|c|}
\hline
\mbox{ghost cluster tilting}\ar@<+2pt>[u]^{F\textrm{-stable}}\ar@<+2pt>[d]\\
\hline
\end{tabular}}\ar@<+2pt>[r]&\ar@<+2pt>[l]^{\textrm{contravariantly finite}\quad}{\begin{tabular}{|c|}
\hline
\mbox{weak ghost cluster tilting}\\
\hline
\end{tabular}}\\
{\begin{tabular}{|c|}
\hline
\mbox{maximal ghost rigid}\ar@<+2pt>[u]^{\textrm{contravariantly finite}}\\
\hline
\end{tabular}}.
}$$

\section{Connection with $\tau$-tilting theory}\label{sect:3}
\setcounter{equation}{0}

Throughout this section, we assume that $\C$ is a $k$-linear, Hom-finite triangulated category with a cluster tilting subcategory
$\T$. It is well-known that the category $\mod\T$ of coherent $\T$-modules is abelian. By Theorem \ref{a5}, we know that the restricted Yoneda functor $\mathbb{H}\colon\C \rightarrow \mod\T$ induces an equivalence
$$  \C / [ \T[1] ] \xrightarrow{~\sim~}  \mod\, \T.$$
In this section, we investigate this relationship between $\C$ and $\mod\, \T$ via $\mathbb{H}$ more closely.

\subsection{On the relationship between ghost cluster tilting and support $\tau$-tilting }
In this subsection, we give a direct connection between ghost cluster tilting subcategories of $\C$ and support $\tau$-tilting pairs of $\mod\T$. We start with the following important observation.

\begin{lem}\label{a6}
Let $\C$ be a triangulated category with a cluster tilting subcategory $\T$ and $\X$ a subcategory of
$\C$.
For any object $X\in\X$, let
\begin{equation}\label{m1}
\begin{array}{l}
T_1\s{f}\longrightarrow T_0 \s{g}\longrightarrow X \s{h}\longrightarrow T_1[1]
\end{array}
\end{equation}
be a triangle in $\C$ with $T_0,T_1\in\T$. Then applying the functor $\mathbb{H}$ gives a projective presentation
\begin{equation}\label{m2}
\begin{array}{l}
P_1^{\mathbb{H}(X)}\xrightarrow{~~~\pi^{\mathbb{H}(X)}}P_0^{\mathbb{H}(X)}\xrightarrow{~~}\mathbb{H}(X)\xrightarrow{~~}0
\end{array}
\end{equation}
in $\mod\, \T$, and $\X$ is a $\T[1]$-rigid subcategory if and only if the class $\{\, \pi^{\mathbb{H}(X)} \,|\, X \in \X \,\}$ has Property (S).
\end{lem}

\proof It is easy to see that $\mathbb{H}$ applies to the triangle (\ref{m1}) gives the projective
presentation (\ref{m2}).\\[0.2cm]
By Theorem \ref{a5}(ii), the map
$\Hom_{\mod\,\T}\big(\pi^{\mathbb{H}(X)},\mathbb{H}(X')\big)$, where $X'\in\X$ is the same as
\begin{equation}\label{t3}
\begin{array}{l}
\Hom_{\C}( T_0, X' )\xrightarrow{\Hom_{\C}(f,\ X')} \Hom_{\C}(T_1,X').
\end{array}
\end{equation}
So the class $\{\, \pi^{\mathbb{H}(X)} \,|\, X\in \X \,\}$ has Property (S) if and
only if the morphism (\ref{t3}) is
surjective for all $X, X'\in\X$.
\vspace{1mm}

Assume that the class $\{\, \pi^{\mathbb{H}(X)} \,|\, X\in \X \,\}$ has Property (S). For any
$a\in[\T[1]](\X,\X[1])$, we know that there exist two morphisms
$a_1\colon X\to T[1]$ and $a_2\colon T[1]\to X'[1]$ such that $a=a_2a_1$, where $X,X'\in\X$ and  $T\in\T$.
Since $\Hom_{\C}(T_0,T[1])=0$, there exists a morphism $b\colon T_1[1]\to T[1]$ such that
$a_1=bh$.
$$\xymatrix@!@C=0.5cm@R=0.3cm{
 T_1\ar[r]^{f} & T_0\ar[r]^{g} & X\ar[r]^{h\quad} \ar[d]^{a_1}& T_1[1]\ar@{.>}[dl]^{b} \\
 &               & T[1]\ar[d]^{a_2} \\
&               & X'[1] }$$
Since $\Hom_{\C}(f,\ X')$ is surjective, there exists a morphism $c\colon T_0\to X'$ such that
$a_2[-1]\circ b[-1]=cf$ and then $a_2b=c[1]\circ f[1]$. It follows that
$a=a_2a_1=a_2bh=c[1]\circ(f[1]h)=0$. This shows that
$[\T[1]](\X,\X[1])=0$. Hence $\X$ is a $\T[1]$-rigid subcategory.
\medskip

Conversely, assume that $\X$ is a $\T[1]$-rigid subcategory. For any morphism
$x\colon T_1\to X'$, since $\X$ is $\T[1]$-rigid, we have $x\circ h[-1]=0$. So there exists a morphism
$y\colon T_0\to X'$ such that $x=yf$.
$$\xymatrix@!@C=0.6cm@R=0.3cm{
 X[-1]\ar[r]^{\; h[-1]}&T_1\ar[r]^{f}\ar[d]^{x}& T_0\ar@{.>}[dl]^{y}\ar[r]^{g} & X\ar[r]^{h\quad}& T_1[1]\\
 &X'&& }$$
This shows that $\Hom_{\C}(f, X')\colon \Hom_{\C}( T_0, X' )\to\Hom_{\C}(T_1,X')$ is surjective. By the above discussion, we know that the class $\{\, \pi^{\mathbb{H}(X)} \,|\, X\in \X \,\}$ has Property (S). \qed
\medskip

The following result plays an important role in this paper.
\begin{thm}\label{a7}
Let $\C$ be a triangulated category with a cluster tilting subcategory $\T$.
The functor $\mathbb{H}\colon \C\to \Mod\, \T$ induces a bijection $$\Phi\colon \X\longmapsto \big(\mathbb{H}(\X),\T\cap\X[-1]\big)$$
from the first of the following sets to the second:
\begin{itemize}
\item[\emph{(I)}] $\T[1]$-rigid subcategories of $\C$.

\item[\emph{(II)}] $\tau$-rigid pairs of $\mod\, \T$.
\end{itemize}
\end{thm}

\proof  \textbf{Step 1:} The map $\Phi$ has values in $\tau$-rigid pairs of $\mod\, \T$.
\medskip

Assume that  $\X$ is a $\T[1]$-rigid subcategory of $\C$.
Since $\T$ is a cluster tilting subcategory, for any $X\in\X$, there exists a triangle in $\C$
$$T_1\s{f}\longrightarrow T_0 \s{g}\longrightarrow X \s{h}\longrightarrow T_1[1],$$
where $T_0,T_1\in\T$.  By Lemma \ref{a6}, we have that
$\mathbb{H}$ sends the set of these triangles to a set of  projective presentations
(\ref{m2}) which has Property (S).
It remains to show that for any $X\in\X$ and $X'\in\T\cap\X[-1]$, we have $\mathbb{H}(X)(X')=0$. Indeed,
since $\X$ is a $\T[1]$-rigid subcategory, we have $\mathbb{H}(X)(X')=\Hom_{\C}(X',X)=0$.
$$\xymatrix{
X'\in\X[-1]\ar[r]\ar@{=}[dr]&X\in\X\\
 &X'\in\T\ar[u] }$$
This shows that $\big(\mathbb{H}(\X),\T\cap\X[-1]\big)$ is a $\tau$-rigid pair of $\mod\, \T$.
\medskip

\textbf{Step 2:} The map $\Phi$ is surjective.
\medskip

Let $(\M,\E)$ be a $\tau$-rigid pair of $\mod\,\T$.  For
each $M\in\M$, take a projective
presentation
\begin{equation}\label{t4}
\begin{array}{l}
P_1\xrightarrow{\pi^M} P_0
\xrightarrow{} M
\xrightarrow{} 0
\end{array}
\end{equation}
such that the class $\{\, \pi^M \,|\, M \in \M \,\}$ has Property
(S). By Theorem \ref{a5}(ii), there is a unique morphism
$f_M\colon T_1 \rightarrow T_0$ in $\T$ such that $\mathbb{H}( f_M ) =
\pi^M$. Moreover, $\mathbb{H}( \mathrm{cone}(f_M) ) \cong M$. Since
(\ref{t4}) has Property (S), it
follows from Lemma \ref{a6} that the category
\[
  \X_1 := \{\, \mathrm{cone}(f_M) \,|\, M \in \M \,\}
\]
is a $\T[1]$-rigid subcategory.
\medskip

 Let $\X:=\X_1\vee\E[1]$.
Now we show that $\X$ is a $T[1]$-rigid subcategory of $\C$.
Let $E\in\E\subseteq\T$. Since $\T$ is cluster-tilting, we have $$[\T[1]](\mathrm{cone}(f_M)\oplus E[1],E[2])=0.$$
Applying the functor $\Hom_{\C}(E,-)$ to the triangle
$T_1\xrightarrow{f_M}T_0\to \mathrm{cone}(f_M)\to T_1[1]$,
we have an exact sequence
\[\Hom_{\C}(E,T_1)\xrightarrow{f_M\circ}\Hom_{\C}(E,T_0)\to\Hom_{\C}(E,\mathrm{cone}(f_M))\to0,\]
which is isomorphic to $$P_1(E)\xrightarrow{\pi^M}P_0(E)\to M(E)\to 0.$$
The condition $\M(\E)=0$ implies that $\Hom_{\C}(E,\mathrm{cone}(f_M))=0$
and then $$[\T[1]](E[1],\mathrm{cone}(f_M)[1])=0.$$
Thus the assertion follows.
\medskip

Now we show that $\Phi( \X ) =( \M , \E )$.
\medskip

It is straightforward to check that $\T\cap\X_1[-1]=0$.
For any object $X\in\T\cap\X[-1]$, we can write $X=X_1[-1]\oplus E\in\T$, where $X_1\in\X_1$ and $E\in\E$.
Since $X_1[-1]\in\T\cap\X_1[-1]=0$, we have $X=E\in\E$. Thus we have $\T\cap\X[-1]\subseteq\E$.
By the definition of $\tau$-rigid pair, we have $\E\subseteq\T$.
Note that $\E\subseteq\X_1[-1]\vee\E=\X[-1]$, it follows that $\E\subseteq\T\cap\X[-1]$.
Hence $\T\cap\X[-1]=\E$.
It remains to show that $\mathbb{H}(\X)=\M$. Indeed, since $\E\subseteq\T$, we have
$$\mathbb{H}(\X)=\Hom_{\C}(\T,\X)=\Hom_{\C}(\T,\X_1)=\mathbb{H}(\X_1)=\M.$$

\textbf{Step 3:} The map $\Phi$ is injective.
\medskip

Let $\X$ and $\X'$ be two $\T[1]$-rigid subcategories of $\C$ such that $\Phi(\X)=\Phi(\X')$. Let
$\X_1$ and $\X'_1$ be respectively the full subcategories of $\X$
and $\X'$ consisting of objects without direct summands in
$\T[1]$. Then $\X=\X_1\vee (\X\cap\T[1])$ and
$\X'=\X'_1\vee (\X'\cap\T[1])$. Since $\Phi(\X)=\Phi(\X')$,
it follows that $\mathbb{H}(\X_1)=\mathbb{H}(\X'_1)$ and
$\X\cap\T[1]=\X'\cap\T[1]$.
\medskip

For any object $X_1\in\X_1$, there exists $X'_1\in\X'_1$ such that $\mathbb{H}(X_1)=\mathbb{H}(X'_1)$. By Theorem \ref{a5}(iii), there exists an isomorphism
$X_1\oplus Y[1]\simeq X'_1\oplus Z[1]$ for some $Y,Z\in\T$. Since $\C$ is Krull-Remak-Schmidt, we have
$X_1\simeq X'_1$. This implies that $\X_1\subseteq \X'_1$. Similarly, we obtain $\X'_1\subseteq \X_1$ and then $\X_1\simeq\X'_1$.
 Therefore $\X=\X'$. This shows that $\Phi$ is injective. \qed
\medskip

Our main result in this subsection is the following.
\begin{thm}\label{a9}
Let $\C$ be a triangulated category with a cluster tilting subcategory $\T$.
The functor $\mathbb{H}\colon \C\to \Mod\, \T$ induces a bijection $$\Phi\colon \X\longmapsto \big(\mathbb{H}(\X),\T\cap\X[-1]\big)$$
from the first of the following sets to the second:
\begin{itemize}
\item[\emph{(I)}] Weak $\T[1]$-cluster tilting subcategories of $\C$.

\item[\emph{(II)}] Support $\tau$-tilting pairs of $\mod\, \T$.
\end{itemize}
\end{thm}

\proof  \textbf{Step 1:} The map $\Phi$ has values in support $\tau$-tilting pairs of $\mod\, \T$.
\medskip

Assume that $\X$ is a weak $\T[1]$-cluster tilting subcategory of $\C$. By Theorem \ref{a7}, we know that $\Phi(\X)$ is a $\tau$-rigid pair of $\mod\, \T$. Therefore $\T\cap\X[-1]\subseteq \Ker\  \mathbb{H}(\X).$
\medskip

Let $T\in\T$ be an object of $\Ker \ \mathbb{H}(\X)$, i.e. $\Hom_{\C}( T,X )=0$ for each
$X\in\X$. This implies that $[\T[1]](X \oplus T[1], \X[1])=0$. Note that $[\T[1]]\big(\X, (X\oplus T[1])[1]\big)=0$. Since $\X$ is a $\T[1]$-cluster tilting subcategory, we have $X\oplus T[1]\in\X$ and then $T\subseteq\X[-1]$. Therefore $T\in \T\cap \X[-1]$.
This shows that $\Ker\  \mathbb{H}(\X) \subseteq \T\cap\X[-1].$ Hence
$$\Ker\ \mathbb{H}(\X) = \T\cap \X[-1].$$

By the definition of weak $\T[1]$-cluster tilting subcategories, for any $T\in\T$, there exists a triangle
$$T\xrightarrow{f}X_1\xrightarrow{g}X_2\xrightarrow{h}T[1],$$
where $X_1,X_2\in\X$.
Applying the functor $\mathbb{H}$ to the above triangle, we obtain an exact sequence
$$
  \mathbb{H}(T)
  \stackrel{\mathbb{H}(f)}{\longrightarrow} \mathbb{H}(X_1)
  \longrightarrow \mathbb{H}(X_2)
  \rightarrow 0.
$$
For any morphism $a\colon T\to X$, where $X \in \X$, since $\X$ is $\T[1]$-rigid, we have $ah[-1]=0$.
So there exists a morphism $b\colon X_1\to X$ such that $a=bf$. This shows that $\Hom_{\C}(f, X)$
is a surjective. Thus there exists the following commutative diagram.
\[
  \xymatrix{
    \Hom_{\C}(X_1,X)\ar[r]^{\Hom_{\C}(f,\ X)}\ar[d] & \Hom_{\C}(T,X)\ar[r] \ar[d] &0\\
    \Hom_{\mod\,\T}(\mathbb{H}(X_1),\mathbb{H}(X))\ar[r]^{\circ\mathbb{H}(f)} & \Hom_{\mod\,\T}(\mathbb{H}(T),\mathbb{H}(X))
           }
\]
Using Theorem \ref{a5}(ii), the right vertical map is an isomorphism. It follows that $\circ\mathbb{H}(f)$ is surjective, that is, $\mathbb{H}(f)$ is a left
$\mathbb{H}(\X)$-approximation.  Altogether, we have shown that $\Phi( \X )$
is a support $\tau$-tilting pair of $\mod\, \T$.
\medskip

\textbf{Step 2:} The map $\Phi$ is surjective.
\medskip

Let $(\M,\E)$ be a support $\tau$-tilting pair of $\mod\,\T$ and let $\X$ be the preimage of $(\M,\E)$ under $\Phi$ constructed in Theorem \ref{a7}. Since $\mathbb{H}(\X)=\M$ is a support $\tau$-tilting subcategory, for each $T\in\T$, there is an exact sequence
$$
  \mathbb{H}(T)
  \stackrel{\alpha}{\rightarrow} \mathbb{H}(X_3)
  \rightarrow \mathbb{H}(X_4)
  \rightarrow 0.
$$
such that $X_3,X_4\in\X$ and $\alpha$ is a left
$\mathbb{H}(\X)$-approximation.  By Yoneda's lemma, there exists a unique morphism
$\beta\colon T\rightarrow X_3$ such that $\mathbb{H}( \beta ) = \alpha$.  Complete it to a triangle
\begin{equation}\label{t5}
\begin{array}{l}
T\s{\beta}\longrightarrow X_3 \s{\gamma}\longrightarrow Y_T \s{\delta}\longrightarrow T[1].
\end{array}
\end{equation}
Let $\widetilde{\X}:= \X\vee\add\{\, Y_T
\,|\, T \in \T \,\}$ be the additive closure of $\X$ and $\{\, Y_T
\,|\, T \in \T \,\}$. We claim that $\widetilde{\X}$ is a weak $\T[1]$-cluster tilting
subcategory of $\C$ such that $\Phi(
\widetilde{\X} ) = ( \M , \E )$.
\medskip

It is clear that $\T\subseteq\X[-1]\ast\X$. \   It remains to show that
$$\widetilde{\X}=\{M\in\C\ |\ [\T[1]](M, \widetilde{\X}[1])=0=[\T[1]](\widetilde{\X}, M[1])\}.$$

Applying the functor $\mathbb{H}$ to the triangle (\ref{t5}), we see that $\mathbb{H}(Y_T)$ and $\mathbb{H}(X_4)$
are isomorphic in $\mod\,\T$. For any object $X\in\X$, consider the following
commutative diagram.
\[
  \xymatrix{
    \Hom_{\C}(X_3,X) \ar[r]^{\Hom_{\C}(\beta,\ X)} \ar[d]_{\mathbb{H}(-)}
      & \Hom_{\C}(T,X)\ar[d]^{\simeq}\\
    \Hom_{\mod\,\T}( \mathbb{H}(X_3) , \mathbb{H}(X) ) \ar[r]^{\circ\alpha}
      & \Hom_{\mod\,\T}(\mathbb{H}(T),\mathbb{H}(X))}
\]
By Theorem \ref{a5}, the map $\mathbb{H}(-)$ is surjective and the right vertical map is an isomorphism.
Because $\alpha$ is a left
$\mathbb{H}(\X)$-approximatiom, $\circ\alpha$ is also surjective. Therefore $\Hom_{\C}(\beta, X)$ is surjective too.
\medskip

For any morphism $a\in[\T[1]](Y_T,X[1])$, since $\X$ is $\T[1]$-rigid, we have
$a\gamma=0$. So there exists a morphism $b\colon T[1]\to X[1]$ such that $a=b\delta$.
$$\xymatrix@!@C=0.5cm@R=0.3cm{
 T\ar[r]^{\beta} & X_3\ar[r]^{\gamma} & Y_T\ar[r]^{\delta\;\;} \ar[d]^{a}& T[1]\ar@{.>}[dl]^{b} \\
 && X[1]&}$$
Since $\Hom_{\C}(\beta,\ X)$ is surjective, there exists a morphism $c\colon X_3\to X$ such that
$c\beta=b[-1]$ and then $b=c[1]\circ \beta[1]$. It follows that
$a=b\delta=c[1]\circ(\beta[1]\delta)=0$.
This shows that
\begin{equation}\label{equi:mainthm1}
\begin{array}{l}
[\T[1]](Y_T,\X[1])=0.
\end{array}
\end{equation}

For any morphism $x\in[\T](X[-1],Y_T)$, we know that there exist two morphisms $x_1\colon X[-1]\to T_1$ and
$x_2\colon T_1\to Y_T$ such that $x=x_2x_1$, where $T_1\in\T$. Since $\T$ is cluster tilting, we have $\delta x_2=0$. So there
exists a morphism $y\colon T_1\to X_3$ such that $x_2=\gamma y$.
$$\xymatrix@!@C=0.5cm@R=0.3cm{
&&X[-1]\ar[d]^{x_1}\\
&&T_1\ar[d]^{x_2}\ar@{.>}[dl]_{y}\\
T\ar[r]^{\beta} & X_3\ar[r]^{\gamma} & Y_T\ar[r]^{\delta\;\;}& T[1]
 }$$
Since $\X$ is $\T[1]$-rigid, we have
$x=x_2x_1=\gamma(yx_1)=0$. This shows that
\begin{equation}\label{equi:mainthm2}
\begin{array}{l}
[\T[1]](\X,Y_T[1])=0.
\end{array}
\end{equation}
For any $T'\in\T$ and morphism $u\in [\T](Y_{T'}[-1],Y_T)$, we know that there exist two morphisms $u_1\colon Y_{T'}[-1]\to T_2$ and
$u_2\colon T_2\to Y_T$ such that $u=u_2u_1$, where $T_2\in\T$.
Since $\T$ is cluster tilting, we have $\delta u_2=0$. So there
exists a morphism $v\colon T_2\to X_3$ such that $u_2=\gamma v$.
$$\xymatrix@!@C=0.5cm@R=0.3cm{
&&Y_{T'}[-1]\ar[d]^{u_1}\\
&&T_2\ar[d]^{u_2}\ar@{.>}[dl]_{v}\\
T\ar[r]^{\beta} & X_3\ar[r]^{\gamma} & Y_T\ar[r]^{\delta\;\;}& T[1]
 }$$
Since $[\T[1]](Y_T,\X[1])=0$, we have $v u_1=0$.
It follows that $u=u_2u_1=\gamma vu_1=0$. This shows that
\begin{equation}\label{equi:mainthm3}
\begin{array}{l}
[\T[1]](Y_{T'},Y_T[1])=0.
\end{array}
\end{equation}
Using equalities (\ref{equi:mainthm1}), (\ref{equi:mainthm2}) and (\ref{equi:mainthm3}), we know that $\widetilde{\X}$ is a $\T[1]$-rigid subcategory.
\medskip

Now we show that $\{M\in\C\ |\ [\T[1]](M, \widetilde{\X}[1])=0=[\T[1]](\widetilde{\X}, M[1])\}\subseteq\widetilde{\X}.$
\medskip

For any object $M\in\C$, assume that $[\T[1]](M, \widetilde{\X}[1])=0=[\T[1]](\widetilde{\X}, M[1])$. Since $\T$ is a cluster-tilting subcategory, there exists a triangle
$$T_5\s{f}\longrightarrow T_6 \s{g}\longrightarrow M \s{h}\longrightarrow T_5[1],$$
where $T_5,T_6\in\T$. By the above discussion, for object $T_6\in\T$, there exists a triangle
$$T_6\s{u}\longrightarrow X_6 \s{v}\longrightarrow Y_{T_6} \s{w}\longrightarrow T_6[1],$$
where $X_6\in\X$, $Y_{T_6}\in\widetilde{\X}$ and $u$ is a left $\X$-approximation of $T_6$.
For object $T_5\in\T$, there exists a triangle
$$T_5\s{u'}\longrightarrow X_5\s{v'}\longrightarrow Y_{T_5} \s{w'}\longrightarrow T_5[1],$$
where $X_5\in\X$, $Y_{T_5}\in\widetilde{\X}$ and $u'$ is a left $\X$-approximation of $T_5$.
By the octahedral axiom, we have a commutative diagram
$$\xymatrix{
T_5\ar[r]^{f}\ar@{=}[d]&T_6\ar[r]^{g}\ar[d]^{u}&M\ar[r]^{h}\ar[d]^{a}&T_5[1]\ar@{=}[d]\\
T_5\ar[r]^{x=uf}&X_6\ar[r]^{y}\ar[d]^{v}&N\ar[r]^{z}\ar[d]^{b}&T_5[1]\\
&Y_{T_6}\ar@{=}[r]\ar[d]^{w}&Y_{T_6}\ar[d]^{c}\\
&T_6[1]\ar[r]^{g[1]}&M[1]}$$
of triangles in $\C$.  We claim that $x$ is a left $\X$-approximation of $T_5$. Indeed, for any $d\colon T_5\to X$,
since $dh[-1]\in[\T](M[-1], \widetilde{\X})=0$, there exists a morphism $e\colon T_6\to X$ such that $d=ef$, where $X\in\X$.
$$\xymatrix{
    M[-1]\ar[r]^{\;\;\, h[-1]}&T_5 \ar[r]^{f} \ar[d]^{d} & T_6\ar@{.>}[dl]^{e} \ar[r]^{g} &M\ar[r]^{h}&T_5[1]\\
    &X&}$$
Since $u$ is a left $\X$-approximation of $T_6$, there exists a morphism $k\colon X_6\to X$ such that $ku=e$. It follows that $d=ef=kuf=kx$, as required.
\medskip

Since $x$ is a left $\X$-approximation of $T_5$, by Lemma 1.4.3 in \cite{ne},  we have the following commutative diagram
$$\xymatrix{
    T_5 \ar[r]^{x} \ar@{=}[d] & X_6\ar[r]^{y}\ar[d]^{\lambda} &N\ar[r]^{z}\ar@{.>}[d]^{\varphi}&T_5[1]\ar@{=}[d]\\
    T_5\ar[r]^{u'}&X_5\ar[r]^{v'}&Y_{T_5}\ar[r]^{w'}&T_5[1],}$$
where the middle square is homotopy cartesian and the differential $\partial=x[1]\circ w'$, that is, there exists a triangle
$$X_6\xrightarrow{\binom{-y}{\lambda}} N\oplus X_5 \xrightarrow{(\varphi,\ v')} Y_{T_5} \xrightarrow{~~\partial~~}X_6[1].$$
Note that $\partial\in[\T[1]](\widetilde{\X},\widetilde{\X}[1])=0$.
Thus we have $N\oplus X_5\simeq X_6\oplus Y_{T_5}\in\widetilde{\X}$, which implies $N\in\widetilde{\X}$.
 Since $c=g[1]w\in [\T[1]](\widetilde{\X},M[1])=0$, we know that the triangle
$$\xymatrix{M\ar[r]^{a}&N\ar[r]^{b}&Y_{T_6}\ar[r]^c&M[1]}$$
splits.  Hence $M$ is a direct summand of $N$ and then $M\in\widetilde{\X}$.
\medskip

This shows that $\widetilde{\X}=\{M\in\C\ |\ [\T[1]](M, \widetilde{\X}[1])=0=[\T[1]](\widetilde{\X}, M[1])\}.$
\vspace{3mm}

For any object $T\in\T$, $\mathbb{H}(Y_T)\simeq \mathbb{H}(X_4)$. Therefore $$\mathbb{H}(\widetilde{\X})\simeq \mathbb{H}(\X)\simeq \M.$$
Since $\T\cap \widetilde{\X}[-1]\supseteq \T\cap \X[-1]=\E$ and $\T\cap \widetilde{\X}[-1]\subseteq
\Ker\ \mathbb{H}(\X)=\E$, we have $$\T\cap \widetilde{\X}[-1]=\E.$$
This shows that $\Phi$ is surjective.
\medskip

\textbf{Step 3:} The map $\Phi$ is injective.
\medskip

This follows from  Step 3 in Theorem \ref{a7}.  \qed
\medskip

For any support $\tau$-tilting subcategory $\Y$ of $\mod\, \T$, by Theorem \ref{a9} there exists a unique weak $\T[1]$-cluster tilting subcategory $\X$ of $\C$ such that $\mathbb{H}(\X)=\Y$. Throughout this paper, we denote the preimage $\X$ by $\mathbb{H}^{-1}(\Y)$ for simplicity. Consequently, we have the following result.
\begin{thm}\label{ghost tilting}
The bijection in Theorem \ref{a9} induces a bijection from the first of the following sets to the second:
\begin{itemize}
\item[\emph{(I)}] $\T[1]$-cluster tilting subcategories of $\C$.

\item[\emph{(II)}] Support $\tau$-tilting subcategories $\Y$ of $\mod\, \T$ such that $\mathbb{H}^{-1}(\Y)$ is contravariantly finite in $\C$.
\end{itemize}
Moveover, if $\C$ admits a Serre functor $\mathbb{S}$, we get a bijection from the first to the second of the following sets.
\begin{itemize}
\item[\emph{(I)}] Cluster tilting subcategories of $\C$.

\item[\emph{(II)}] Support $\tau$-tilting subcategories $\Y$ of $\mod\, \T$ such that $\mathbb{H}^{-1}(\Y)$ is contravariantly finite and $F$-stable in $\C$ .
\end{itemize}
\end{thm}
\begin{proof} The first bijection follows from $\T[1]$-cluster tilting subcategories of $\C$ are precisely contravariantly finite weak $\T[1]$-cluster tilting subcategories, and the second bijection follows from Theorem \ref{thm:F-stable}.
\end{proof}

\subsection{$\tau$-tilting subcategories and tilting subcategories}
In this subsection, we assume that $\C$ is   a triangulated category with a cluster tilting subcategory $\T$.
 By definition we know that the category $\mod\T$ is abelian and has enough projectives. Thus we can investigate the projective dimension of an object $M$ in $\mod\T$, which we denote by $\pd M$. For a subcategory $\D$ of $\mod\T$, we say that the projective dimension of $\D$ is at most $n$, denoted by $\pd\D\leq n$, if $\pd M\leq n$ for any object $M\in\D$.
\medskip

Let $X\in \C$, $\I_{X}(\T[1])$ be the ideal of $\T[1]$ formed by the morphisms
between objects in $\T[1]$ factoring through the object $X$. For a subcategory $\D$ of $\C$, we define the \emph{factorization ideal }of $\D$, denoted by $\I_{\D}(\T[1])$, as follows
$$\I_{\D}(\T[1]):= \{\ \I_{X}(\T[1])\ |\ X\in\D\ \}.$$
Theorem \ref{a5} indicates that $\mod\, \T$ is Gorenstein of dimension at most one. Thus all objects in $\mod\, \T$ have projective dimension zero, one or infinity. The following result characterizes the objects in $\mod\T$ having finite projective dimension.

\begin{thm}\emph{\cite{BBT,la}}\label{factor}
Let $\C$ be a triangulated category with a cluster tilting subcategory $\T$, and $X$ be an object in $\C$ having no direct summands in $\T[1]$. Then $$\pd \mathbb{H}(X)\leq 1\ if\  and\ only\ if\  \I_{X}(\T[1])=0.$$
\end{thm}
In this subsection, we introduce two important classes of subcategories of $\mod\T$ and give a connection with ghost cluster tilting subcategories and cluster tilting subcategories of $\C$. We start with the following definition.

\begin{defn}\label{tilting}
Let $\C$ be a triangulated category with a cluster tilting subcategory $\T$.
\begin{enumerate}

\item[\emph{(i)}] A subcategory $\M$ of $\mod\, \T$ is said to be
  \emph{$\tau$-tilting} if $(\M, 0)$ is a support $\tau$-tilting pair of $\mod\T$.
\item[\emph{(ii)}]  \emph{\cite{Be2}} A subcategory $\M$ of $\mod\, \T$ is said to be
  \emph{weak tilting} if the following three conditions are satisfied:

\begin{enumerate}
\item[\emph{(T1)}] $\emph{\Ext}^1_{\mod\T}(\M,\M)=0.$
\item[\emph{(T2)}] $\pd M\leq 1$, for any $M\in\M$.
\item[\emph{(T3)}] for any projective object $P$ in $\mod\T$, there exists a
short exact sequence
$$0\s{}\longrightarrow P \s{}\longrightarrow M_0 \s{}\longrightarrow M_1\s{}\longrightarrow 0$$
where $M_0,M_1\in\M$.
\end{enumerate}
A weak tilting subcategory $\M$ is called a \emph{tilting subcategory} if it also satisfies the
following additional condition:
\begin{enumerate}
\item[\emph{(T4)}] $\M$ is contravariantly finite in $\mod\T$.
\end{enumerate}

\end{enumerate}
\end{defn}

\begin{rem}
Beligiannis \emph{\cite{Be1,Be2}} indicates that a contravariantly finite subcategory $\M$ of $\mod\T$ is a tilting subcategory if and only if
$$\Fac(\M) = \{\ X\in\mod\T\ |\ \text{\emph{Ext}}^1_{\mod\T}(\M, X) = 0\ \},$$
where $\Fac(\M)$ is the full subcategory of $\mod\T$ consisting of all factors of objects of $\M$.
\end{rem}
Immediately, we have the following result.
\begin{thm}\label{thm5}
Let $\C$ be a triangulated category with a cluster tilting subcategory $\T$.
The functor $\mathbb{H}\colon \C\to \Mod\, \T$ induces a bijection $$\Phi\colon \X\longmapsto \mathbb{H}(\X)$$
from the first of the following sets to the second:
\begin{itemize}
\item[\emph{(I)}] Weak $\T[1]$-cluster tilting subcategories of $\C$ whose objects do not have non-zero direct summands in $\T[1]$.

\item[\emph{(II)}] $\tau$-tilting subcategories of $\mod\,\T$.
\end{itemize}
It restricts to a bijection from the first to the second of the following sets.
\begin{itemize}
\item[\emph{(I)}] $\T[1]$-cluster tilting subcategories of $\C$ whose objects do not have non-zero direct summands in $\T[1]$.

\item[\emph{(II)}] $\tau$-tilting subcategories $\Y$ of $\mod\, \T$ such that $\mathbb{H}^{-1}(\Y)$ is contravariantly finite in $\C$.
\end{itemize}
Moreover, if $\C$ admits a Serre functor $\mathbb{S}$, we get the following bijection.
\begin{itemize}
\item[\emph{(I)}] Cluster tilting subcategories of $\C$ whose objects do not have non-zero direct summands in $\T[1]$.
\item[\emph{(II)}] $\tau$-tilting subcategories $\Y$ of $\mod\, \T$ such that $\mathbb{H}^{-1}(\Y)$ is contravariantly finite and $F$-stable in $\C$ .
\end{itemize}
\end{thm}

\proof Note that objects in $\X$ do not have non-zero direct summands in $\T[1]$ if and only if $\T\cap\X[-1] =0$. This assertion follows from Theorem \ref{a9} and Theorem \ref{ghost tilting} directly.    \qed
\medskip

Now we give a close relationship between $\tau$-tilting subcategories and weak tilting subcategories.

\begin{lem}\label{lem4}
Let $\C$ be a triangulated category with a cluster tilting subcategory $\T$. Then
any weak tilting subcategory of $\mod\T$ is a $\tau$-tilting subcategory.
\end{lem}

\proof Let $\M$ be a weak tilting subcategory of $\mod\T$.
\medskip

(1) We first show that $(\M, 0)$ is a $\tau$-rigid pair of $\mod\T$. For any object $M\in \M$, since $\pd M\leq 1$, we get a short exact sequence
$$0\longrightarrow P_1 \stackrel{ \pi^M }{ \longrightarrow } P_0 \longrightarrow M \longrightarrow 0.$$
   Note that $P_1=0$ if $\pd M=0$. Applying the functor $\Hom_{\mod\T}(-,\M)$ to it, we get an exact sequence
$$\Hom_{\mod\T}(P_0,\M)\xrightarrow{\circ \pi^M}\Hom_{\mod\T}(P_1,\M)\xrightarrow{}\Ext^1_{\mod\T}(M,\M)=0.$$
This means that there is a class of projective presentations
  $\{P_1 \stackrel{ \pi^M }{\rightarrow } P_0 \rightarrow M \rightarrow 0\mid M\in\M\}$
  which has Property (S). Therefore $(\M, 0)$ is a $\tau$-rigid pair of $\mod\T$ because $\M(0)=0$.
\medskip

(2) We show that $(\M, 0)$ is a support $\tau$-tilting pair of $\mod\T$. For each object $T\in\T$, $\T(-, T)$ is a projective object in $\mod\T$. Since $\M$ is weak tilting in $\mod\T$, there exists a
short exact sequence
$$0\s{}\longrightarrow \T(-, T) \s{f}\longrightarrow M_0 \s{}\longrightarrow M_1\s{}\longrightarrow 0$$
where $M_0,M_1\in\M$.
Applying the functor $\Hom_{\mod\T}(-,\M)$ to the above exact sequence, we have the following exact sequence:
$$\Hom_{\mod\T}(M_0,\M)\xrightarrow{\circ f}\Hom_{\mod\T}(\T(-, T),\M)\xrightarrow{}\Ext^1_{\mod\T}(M_1,\M)=0.$$
This shows that $f$ is a left $\M$-approximation.
\medskip

If $\M(E)=0$, where $E\in\T$, by the above discussion,  there exists an exact sequence $$0\s{}\longrightarrow \T(-,E) \s{}\longrightarrow M_0 \s{}\longrightarrow M_1\s{}\longrightarrow 0$$ with
  $M^0, M^1 \in \M$. It follows that there exists an exact sequence
  $$0\s{}\longrightarrow \T(E,E) \s{}\longrightarrow M_0(E) \s{}\longrightarrow M_1(E)\s{}\longrightarrow 0$$
Since $M_0(E)=0$, we have $\T(E,E)=0$ and then $E=0$. Therefore Ker$\,( \M )=0$.
This shows that $(\M, 0)$ is a support $\tau$-tilting pair of $\mod\T$.
 \qed
\medskip

The following result gives a criterion for a $\tau$-tilting subcategory of $\mod\T$ to be a weak tilting subcategory.
\medskip

\begin{thm}\label{thm6}
Let $\C$ be a triangulated category with a cluster tilting subcategory $\T$. A $\tau$-tilting subcategory of $\mod\T$ is a weak tilting subcategory if and only if its projective dimension is at most one.
\end{thm}

\proof Let $\M$ be a $\tau$-tilting subcategory of $\mod\T$ and $\pd\M\leq 1$. By Theorem \ref{thm5}, there exists a weak $\T[1]$-tilting subcategory $\X$ of $\C$  whose objects do not have non-zero direct summands in $\T[1]$ such that $\mathbb{H}(\X)=\M$.
\medskip

\textbf{Step 1:} We show that $\Ext^1_{\mod\T}(\M,\M)=0.$ Namely, $\Ext^1_{\mod\T}(\mathbb{H}(\X),\mathbb{H}(\X))=0.$
\vspace{1mm}

For any object $X_1\in\X$, since $\T$ is cluster tilting, there exists a triangle
\begin{equation}\label{triangle}
\begin{array}{l}
\xymatrix{T_0\ar[r]^{f}&T_1\ar[r]^{g}&X_1\ar[r]^{h\;\,}&T_0[1]},\end{array}
\end{equation}
where $g$ is a minimal right $\T$-approximation of $X_1$ and $T_0,T_1\in\T$.
Since $\mathbb{H}(X_1)\in\M$, we have $\pd \mathbb{H}(X_1)\leq1$.
Applying the functor $\mathbb{H}$ to the above triangle, we have a minimal projective presentation
$$\xymatrix{0\ar[r]&\mathbb{H}(T_0)\ar[r]^{f\circ}&\mathbb{H}(T_1)\ar[r]^{g\circ}&\mathbb{H}(X_1)\ar[r]& 0}$$
of $\mathbb{H}(X_1)$, since $X_1$ has no non-zero direct summands in $\T[1]$ and $\pd \mathbb{H}(X_1)\leq1$.
Applying the functor $\Hom_{\mod\T}(-,\mathbb{H}(X_2))$, where $X_2\in\X$ to the above exact sequence, we get an exact sequence:
$$\xymatrix{\Hom_{\mod\T}(\mathbb{H}(T_1),\mathbb{H}(X_2))\ar[r]&\Hom_{\mod\T}(\mathbb{H}(T_0),\mathbb{H}(X_2))\ar[d]\\
&\Ext^1_{\mod\T}(\mathbb{H}(X_1),\mathbb{H}(X_2))\ar[r]&\Ext^1_{\mod\T}(\mathbb{H}(T_1),\mathbb{H}(X_2))=0.}$$
The last item vanishes because $\mathbb{H}(T_1)$ is projective in $\mod\T$. Note that the first map is isomorphic to
$\Hom_{\C}(T_1,X_2)\xrightarrow{\Hom_{\C}(f,\, X_2)}\Hom_{\C}(T_0,X_2)$. This follows that
$\Ext^1_{\mod\T}(\mathbb{H}(X_1),\mathbb{H}(X_2))$ is isomorphic to $\co\Hom_{\C}(f,X_2)$.
\medskip

Applying the functor $\Hom_{\C}(-, X_2)$ to the triangle (\ref{triangle}), we have the following exact sequence:
$$\Hom_{\C}(T_1,X_2)\xrightarrow{\Hom_{\C}(f,\, X_2)}\Hom_{\C}(T_0,X_2)\xrightarrow{\Hom_{\C}(h[-1],\,X_2)}
\Hom_{\C}(X_1[-1],X_2).$$
In particular, we have the following exact sequence:
$$\Hom_{\C}(T_1,X_2)\xrightarrow{\Hom_{\C}(f,\, X_2)}\Hom_{\C}(T_0,X_2)\xrightarrow{\Hom_{\C}(h[-1],\,X_2)}
\Im\,\Hom_{\C}(h[-1],\,X_2)\xrightarrow{} 0.$$
We claim that $\Im\,\Hom_{\C}(h[-1],\,X_2)=[\T](X_1[-1],X_2)$. Indeed, $$\Im\,\Hom_{\C}(h[-1],\,X_2)\subseteq[\T](X_1[-1],X_2)$$ is clear. For any morphism $x\in[\T](X_1[-1],X_2)$, we know that
there exist two morphisms $x_1\colon X_1[-1]\to T$ and $x_2\colon T\to X_2$, where $T\in\T$ such that $x=x_2x_1$.
Since $\Hom_{\C}(T_1[-1],T)=0$, there exists a morphism $a\colon T_0\to T$ such that $ah[-1]=x_1$.
$$\xymatrix{
    T_1[-1]\ar[dr]^{0}\ar[r]^{ g[-1]}&X_1[-1] \ar[r]^{\quad h[-1]} \ar[d]^{x_1} & T_0\ar@{.>}[dl]^{a} \ar[r]^{f} &T_1\ar[r]^{g}& X_1\ar[r]^{h}&T_0[1]\\
    &T&}$$
It follows that $x=x_2x_1=(x_2a)h[-1]\in\Im\,\Hom_{\C}(h[-1],\,X_2)$, as required.
\medskip

Since $\X$ is $\T[1]$-rigid, we have $[\T](X_1[-1],X_2)=0$. Thus we obtain
$$\Ext^1_{\mod\T}(\mathbb{H}(X_1),\mathbb{H}(X_2))\simeq\co\Hom_{\C}(f,X_2)=[\T](X_1[-1],X_2)=0.$$

\textbf{Step 2:} We show that for any projective object $P$ in $\mod\T$, there exists a
short exact sequence
$$0\s{}\longrightarrow P \s{}\longrightarrow M_0 \s{}\longrightarrow M_1\s{}\longrightarrow 0$$
where $M_0,M_1\in\M$. We may assume that $P=\T(-, T)=\mathbb{H}(T)$ in $\mod\T$, where $T\in\T$.
\medskip

Since $T\in\T\subseteq\X[-1]\ast\X$, there exists a triangle
$$\xymatrix{X_3[-1]\ar[r]^{\quad u}&T\ar[r]^{v}&X_4\ar[r]^{w\;}&X_3},$$
where $X_3,X_4\in\X$. Applying the functor $\mathbb{H}$ to the above triangle, we have the following exact sequence:
$$\mathbb{H}(X_3[-1])\xrightarrow{\mathbb{H}(u)}\mathbb{H}(T)\xrightarrow{}
\mathbb{H}(X_4)\xrightarrow{}\mathbb{H}(X_3)\xrightarrow{}0.$$

We claim that $\Im\,\mathbb{H}(u)=0$. That is to say, for any morphism $y\colon T'\to X_3[-1]$, where $T'\in\T$, we
have $uy=0$. Indeed, since $\T$ is cluster tilting, there exists a triangle
$$\xymatrix{T_2\ar[r]^{\alpha}&T_3\ar[r]^{\beta}&X_3\ar[r]^{\gamma\;\,}&T_2[1]},$$
where $\beta$ is a minimal right $\T$-approximation of $X_3$ and $T_2, T_3\in\T$.
Applying the functor $\mathbb{H}$ to the above triangle, we have a minimal projective presentation
$$\mathbb{H}(X_3[-1])\xrightarrow{\mathbb{H}(\gamma[-1])}\mathbb{H}(T_2)\xrightarrow{\mathbb{H}(\alpha)}\mathbb{H}(T_3)\xrightarrow{\mathbb{H}(\beta)}\mathbb{H}(X_3)\xrightarrow{}0.$$
of $\mathbb{H}(X_3)$, since $X_1$ has no non-zero direct summands in $\T[1]$.
Since $\mathbb{H}(X_3)\in\M$, we have $\pd \mathbb{H}(X_3)\leq1$.
Thus we have $\Im\,\mathbb{H}(\gamma[-1])=0$ and then $\gamma[-1]\circ y=0$. So there exists a morphism
$b\colon T'\to T_3[-1]$ such that $y=\beta[-1]\circ b$.
$$\xymatrix@!@C=1cm@R=0.5cm{
&T'\ar[d]^{y}\ar@{.>}[dl]_{b}\\
 T_3[-1]\ar[r]^{\beta[-1]} & X_3[-1]\ar[r]^{\quad\gamma[-1]}& T_2\ar[r]^{\alpha}&T_3.}$$
It follows that $uy=(u\beta[-1])b=0\circ b=0$, as required.
Hence  we have the following exact sequence:
$$0\xrightarrow{}\mathbb{H}(T)\xrightarrow{}
\mathbb{H}(X_4)\xrightarrow{}\mathbb{H}(X_3)\xrightarrow{}0,$$
where $\mathbb{H}(X_4),\mathbb{H}(X_3)\in\M$.
\medskip

This shows that $\M$ is a weak tilting subcategory of $\mod\T$. Combining with Lemma \ref{lem4}, the assertion follows.\qed
\medskip

Consequently, we have the following result.
\begin{thm}\label{thm7}
Let $\C$ be a triangulated category with a cluster tilting subcategory $\T$.
The functor $\mathbb{H}\colon \C\to \Mod\, \T$ induces a bijection $$\Phi\colon \X\longmapsto \mathbb{H}(\X)$$
from the first of the following sets to the second:
\begin{itemize}
\item[\emph{(I)}] Weak $\T[1]$-cluster tilting subcategories of $\C$ whose objects do not have non-zero direct summands in $\T[1]$ and whose factorization ideals vanish.
\item[\emph{(II)}] Weak tilting subcategories of $\mod\,\T$.
\end{itemize}
It restricts to a bijection from the first to the second of the following sets.
\begin{itemize}
\item[\emph{(I)}] $\T[1]$-cluster tilting subcategories of $\C$ whose objects do not have non-zero direct summands in $\T[1]$ and whose factorization ideals vanish.
\item[\emph{(II)}] Tilting subcategories $\Y$ of $\mod\, \T$ such that $\mathbb{H}^{-1}(\Y)$ is contravariantly finite in $\C$.
\end{itemize}
Moreover, if $\C$ admits a Serre functor $\mathbb{S}$, we get the following bijection.
\begin{itemize}
\item[\emph{(I)}] Cluster tilting subcategories of $\C$ whose objects do not have non-zero direct summands in $\T[1]$ and whose factorization ideals vanish.
\item[\emph{(II)}] Tilting subcategories $\Y$ of $\mod\, \T$ such that $\mathbb{H}^{-1}(\Y)$ is contravariantly finite and $F$-stable in $\C$ .
\end{itemize}
\end{thm}
\proof This follows from Theorem \ref{factor},  Theorem \ref{thm5} and Theorem \ref{thm6} directly. \qed

\begin{rem}
The above results generalize and improve several results in the literature. More
precisely, Theorem \ref{a7},  Theorem \ref{a9}, Theorem \ref{thm5} and Theorem \ref{thm7}
generalize a result of Yang-Zhu, see \cite[Theorem 3.6]{YZ}, where analogous results
were proved in case $\C$ is $2$-Calabi-Yau and $\T=\add T$ \cite[Theorem 4.1]{AIR}. Theorem \ref{thm7} generalizes a  result of Beligiannis \cite[Theorem 6.6]{Be2} in some cases, but we don't assume that $\mod\X$ has finite global dimension here.
\end{rem}

\section{A characterization of ghost cluster tilting objects}
\setcounter{equation}{0}

In this section, we always assume that $\C$ is a triangulated category with a Serre functor and with cluster tilting objects.
\medskip

In Section \ref{sect:1}, we have given the following notion.

\begin{defn}\label{y5}
An object $X$ in $\C$ is called {\rm ghost cluster tilting} if  there exists a cluster tilting object $T$ such that
$$\add X=\{M\in\C\ |\ [T[1]](M, X[1])=0=[T[1]](X, M[1])\}.$$
In this case, $X$ is also called $T[1]${\rm -cluster tilting}.
\end{defn}

We recall the definition of relative cluster tilting objects from \cite{YZ}.

\begin{defn}\label{y1}{\emph{\cite[Definition 3.1]{YZ}}}
An object $X$ in $\C$ is called {\rm relative cluster tilting} if there exists a cluster tilting object $T$ such that $X$ is $T[1]$-rigid and $|X|=|T|(=r(\C))$. In this case, $X$ is also called \emph{$T[1]$-cluster tilting}.
\end{defn}

The aim of this section is to show that these two definitions are equivalent. The following analog of Lemma \ref{lem:main} is crucial in this section.
\begin{lem}\label{lem:main2}
 Let $X$ be a $T[1]$-cluster tilting object in Definition \ref{y1}. For any $T_0\in\add T$, if $$M[-1]\s{f}\longrightarrow T_0 \s{g}\longrightarrow X_0 \s{h}\longrightarrow M$$ is a triangle such that $X_0\in\add X$ and $g\colon T_0\longrightarrow X_0$ is a left $\add X$-approximation of $T_0$, then $M\in \add X$. Dually, if
 $$X_0[-1]\s{f}\longrightarrow T_0 \s{g}\longrightarrow M\s{h}\longrightarrow X_0$$ is a triangle such that $X_0\in\add X$ and $f\colon X_0[-1]\longrightarrow T_0$ is a right $\add X[-1]$-approximation of $T_0$, then $M\in \add X$.
\end{lem}
\begin{proof} Using similar arguments as in the proof of Lemma \ref{lem:main} we conclude that
$$[T[1]](X\oplus M, (X\oplus M)[1])=0.$$ By Corollary 3.7(1) in \cite{YZ}, we have $|X\oplus M|\leq |T|$. Since $|X|=|T|$, we know that $M\in \add X$.
\end{proof}

The following result was proved in \cite[Corollary 3.7(2)]{YZ}. For the convenience of the readers, now we give a triangulated version of Bongartz's classical proof.

\begin{lem}\label{h1}
Let $\C$ be a triangulated category with a Serre functor and a cluster tilting object $T$, and let $\la={\rm End}^{\emph{op}}_{\C}(T)$. Then
any $T[1]$-rigid object in $\C$ is a direct summand of some $T[1]$-cluster tilting object in the sense of Definition \ref{y1}.
\end{lem}

\proof Assume that $X$ is a $T[1]$-rigid object, we take a triangle
$$T\s{f}\longrightarrow U \s{g}\longrightarrow X_1 \s{h}\longrightarrow T[1],$$
where $h$ is a  right $(\add X)$-approximation of $T[1]$.
We claim that $R:=U\oplus X$ is a $T[1]$-cluster tilting object in the sense of Definition \ref{y1}.
\medskip

\textbf{Step 1:} We show that $R=U\oplus X$ is a $T[1]$-rigid object.
\medskip

Take any map $a\in [T[1]](X, U[1])$. Since $a$ factors through $\add T[1]$ and $X$ is $T[1]$-rigid, we have $g[1]a=0$.
$$\xymatrix@!@C=0.8cm@R=0.6cm{
 U\ar[r]^g & X_1\ar[r]^h    & T[1]\ar[r]^{-f[1]\;} & U[1]\ar[r]^{-g[1]\;}  &  X_1[1]   \\
   &&&X\ar[u]_{a}\ar@{.>}[ul]_{b}\ar@{.>}[ull]^{c}              }
$$
Thus there exists a morphism $b\colon X\rightarrow T[1]$ such that $a=-f[1]b$.
Since $h$ is a right $(\add X)$-approximation of $T[1]$,  there exists a morphism $c\colon X\rightarrow X_1$ such that $b=hc$. It follows that $$a=-f[1]b=(-f[1]h)c=0,$$ and therefore $$[T[1]](X, U[1])=0.$$

For any morphism $u\in [T[1]](U, X[1])$, we know that there are two morphisms $u_1\colon U\rightarrow T_0[1]$ and $u_2\colon T_0[1]\rightarrow X[1]$ such that $u=u_2u_1$, where $T_0\in \add T$.
$$\xymatrix@!@C=0.6cm@R=0.6cm{
X_1[-1]\ar[r]^{\quad h[-1]}&T\ar[r]^f&U\ar[d]^{u_1} \ar[r]^g & X_1\ar@{.>}[ld]^{v}   \\
   && T_0[1]\ar[d]^{u_2}\\
   && X[1]             }
$$
Since $\Hom_{\C}(T,T_0[1])=0$, there exists a morphism $v\colon X_1\to T_0[1]$ such that $u_1=vg$.
Because $X$ is $T[1]$-rigid, we have $u=u_2u_1=(u_2v)g=0$. Therefore $$[T[1]](U, X[1])=0.$$

For any morphism $x\in [T[1]](U, U[1])$, we know that there are two morphisms $x_1\colon U\rightarrow T_1[1]$ and $x_2\colon T_1[1]\rightarrow U[1]$ such that $x=x_2x_1$, where $T_1\in \add T$.
$$\xymatrix@!@C=0.6cm@R=0.6cm{
T\ar[r]^{f} & U\ar[r]^g\ar[d]^{x_1} & X_1\ar[r]^h\ar@{.>}[ld]^{y}& T[1]\\
   & T_1[1]\ar[d]^{x_2}\\
   & U[1]             }$$
Since $\Hom_{\C}(T,T_1[1])=0$, there exists $y\colon X_1\to T_1[1]$ such that $x_1=yg$.
Because $x_2y\in[T[1]](X_1,U[1])$ and $[T[1]](X,U[1])=0$, we have $x=x_2x_1=(x_2y)g=0$. Therefore $$[T[1]](U, U[1])=0.$$
This shows that $R=U\oplus X$ is a $T[1]$-rigid object.
\medskip

\textbf{Step 2:} We show that $f$ is a left $(\add R)$-approximation of $T$.
\medskip

For any morphism $\alpha\colon T\to R'$, where $R'\in\add R$, by Step 1, we have
$[T[1]](X,R[1])=0$. Thus $\alpha[1]h=0$ and then $\alpha h[-1]=0$. We obtain a commutative diagram
$$\xymatrix{
X_1[-1]\ar[r]^{\quad h[-1]}&T\ar[d]^{\alpha}\ar[r]^f&U\ar@{.>}[ld]^{\beta}\ar[r]^g & X_1  \\
&R'&}$$
Hence $f$ is a left $(\add R)$-approximation of $T$.
\medskip

\textbf{Step 3:} We show that $R=U\oplus X$ is a $T[1]$-cluster tilting object in the sense of Definition \ref{y1}.
\medskip

Applying the functor $\Hom_{\C}(T,-)$ to this triangle
$$T\s{f}\longrightarrow U \s{g}\longrightarrow X_1 \s{h}\longrightarrow T[1],$$
we have the following exact sequence
$$\Lambda\s{\overline{f}}\longrightarrow \overline{U} \longrightarrow \overline{X_1 }\longrightarrow 0,$$
where $\overline{U},\overline{X_1}\in\add \overline{R}$.
By the equivalence (\ref{w1}) and Step2, we have that $\overline{f}$ is a left $(\add \overline{R})$-approximation of $\Lambda$.
By Theorem \ref{y3} and Step1, we have that $\overline{R}$ is a $\tau$-rigid module.
By Lemma \ref{y4}, we know that $\overline{R}$ is a support $\tau$-tilting $\Lambda$-module.
By Theorem \ref{y3}, we have that $R=U\oplus X$ is a $T[1]$-cluster tilting object
in the sense of Definition \ref{y1}.\qed

\medskip

Our main result in this section is the following.

\begin{thm}\label{thm:equi}
Let $\C$ be a triangulated category with a Serre functor and a cluster tilting object $T$,  and let $\la={\rm End}^{\emph{op}}_{\C}(T)$. Then
Definition \ref{y5} and  Definition \ref{y1} are equivalent.
\end{thm}

\proof (1)\ \ Definition \ref{y5} $\Rightarrow$ Definition \ref{y1}.
\medskip

Assume that $X$ is a $T[1]$-cluster tilting object in the sense of Definition \ref{y5}. Then $X$ is $T[1]$-rigid.
By Lemma \ref{h1}, there exists an object $M\in\C$ such that $X\oplus M$ is $T[1]$-cluster tilting in the sense of Definition \ref{y1}. That is to say, $X\oplus M$ is $T[1]$-rigid and $|X\oplus M|=|T|$.
Since $X\oplus M$ is $T[1]$-rigid, we have $$[T[1]](M, X[1])=0=[T[1]](X, M[1]).$$ By Definition \ref{y5}, we have
$M\in\add X$. It follows that $|X|=|X\oplus M|=|T|$. This shows that $X$ is $T[1]$-cluster tilting in the sense of Definition \ref{y1}.
\medskip

(2) \ \ Definition \ref{y1} $\Rightarrow$ Definition \ref{y5}.
   \medskip

   Assume that $X$ is a $T[1]$-cluster tilting object in Definition \ref{y1}. Clearly, $$\add X\subseteq \{M\in\C\ |\ [T[1]](X, M[1])=0= [T[1]](M, X[1])\}.$$
Conversely, for any object $M\in \{M\in\C\ |\ [T[1]](X, M[1])=0= [T[1]](M, X[1])\}$, since $T$ is cluster tilting, we have a triangle $$T_1\s{f}\longrightarrow T_0 \s{g}\longrightarrow M\s{h}\longrightarrow T_1[1],$$ where $T_0, T_1\in\add T.$ By Lemma \ref{lem:main2}, there is a left $\add X$-approximation $l_1$ of $T_0$ which extends to a triangle
$$X_1[-1]\s{m}\longrightarrow T_0 \s{l_1}\longrightarrow X_0 \longrightarrow X_1,$$ where $X_0, X_1\in\add X.$
\medskip

Let $l_2=l_1f$. It is easy to see that $l_2$ is a left $\add X$-approximation of $T_1$. Indeed, for any object $X'\in \add X$ and any map $a\in\Hom(T_1, X')$, we have $ah[-1]\in [T](M[-1], X')=0$. Then there exists $b\colon T_0\longrightarrow X'$ such that $a=bf$. Because $l_1$ is a left $\add X$-approximation of $T_0$, there is a map $c\colon X_0\longrightarrow X'$ such that $b=cl_1$. Therefore $a=bf=c(l_1f)=cl_2$ and $l_2$ is a left ($\add X$)-approximation of $T_1$.
\begin{center}
$\xymatrix@!@C=0.5cm@R=0.4cm{
 M[-1]\ar[r]^{\quad h[-1]}  &  T_1\ar[r]^{f}\ar[d]_{\forall a} & T_0\ar[r]^{g}\ar@{.>}[dl]^{b}\ar[d]^{l_1} & M\ar[r]^{h\quad} & T_1[1] \\
                      &  X'                      & X_0\ar@{.>}[l]^{c}
  }$
\end{center}
Using Lemma \ref{lem:main2}, we get a triangle $$X_2[-1] \longrightarrow T_1 \s{l_2}\longrightarrow X_0 \longrightarrow X_2,$$ where $X_2\in\add X.$
Starting with $l_2=l_1f$, we get the following commutative diagram by the octahedral axiom.
\begin{center}
$\xymatrix@!@C=0.5cm@R=0.5cm{
               &M[-1]\ar[d]\ar@{=}[r]  & M[-1]\ar[d]  \\
X_0[-1]\ar@{=}[d]\ar[r]  & X_2[-1]\ar[d]\ar[r]     & T_1\ar[r]^{l_2}\ar[d]^{f} & X_0\ar@{=}[d]  \\
X_0[-1]\ar[r]            & X_1[-1]\ar[d]^n\ar[r]^{\quad m}     & T_0\ar[d]^g\ar[r]^{l_1}    & X_0 \\
               & M\ar@{=}[r]      & M
  }$
\end{center}
Since $n=gm\in [T](X_1[-1], M)=0$, we get a split triangle and then $M\in\add X$.
This shows that $X$ is a $T[1]$-cluster tilting object in the sense of Definition \ref{y5}.   \qed
\medskip

We get the following direct consequence.
\begin{cor}
If $\C$ is a triangulated category with a Serre functor and a cluster tilting object, then relative cluster tilting objects are precisely maximal ghost rigid objects.
\end{cor}
\proof This follows from Theorem \ref{a3} and Theorem \ref{thm:equi} directly.
\qed

\begin{rem}
Let $\C$ be a triangulated category with a Serre functor and a cluster tilting object $T$. One may try to define ghost cluster tilting objects in $\C$ as follows (two analogues of the definition of cluster tilting objects):
\begin{equation}\label{d1}
\begin{array}{l}
\add X= \{M\in\C\ |\ [T[1]](X, M[1])=0\},
\end{array}
\end{equation} or
\begin{equation}\label{d2}
\begin{array}{l}
\add X=\{M\in\C\ |\ [T[1]](X, M[1])=0\}= \{M\in\C\ |\ [T[1]](M, X[1])=0\}.
\end{array}
\end{equation}
 However, we have the following example to illustrate that (\ref{d1}) or (\ref{d2}) is not equivalent to Definition \ref{y1}.
\end{rem}
\begin{exm}
{\upshape Let $Q$ be the quiver $\setlength{\unitlength}{0.03in}\xymatrix{1 \ar[r]^{\alpha} & 2}$ and $\tau_Q$ be the Auslander-Reiten translation in $D^b(kQ)$. We consider a triangulated category, named \emph{repetitive cluster category} in \cite{Zh}, $\C=D^b(kQ)/\langle \tau_Q^{-2}[2]\rangle$, whose objects are the same in $D^b(kQ)$, and whose morphisms are given by $$\Hom_{D^b(kQ)/\langle \tau_Q^{-2}[2]\rangle}(X, Y)=\bigoplus_{i\in \mathbb{Z}}\Hom_{D^b(kQ)}(X, (\tau_Q^{-2}[2])^iY).$$  }
\end{exm}
We depict the AR-quiver of $\C$ as follows.
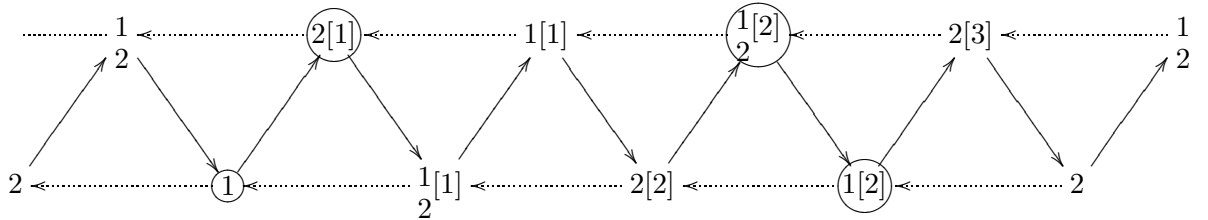
\begin{center}\label{f:1}$
\xymatrix@!@C=0.4cm@R=1cm{  
&\txt{1\\2}\ar@{.}[l]\ar[dr]&&*+[o][F-]{2[1]}\ar@{.>}[ll]\ar[dr]&&1[1]\ar@{.>}[ll]\ar[dr]&&*+[o][F-]{\txt{1\\2}[2]}\ar@{.>}[ll]\ar[dr]&&2[3]\ar@{.>}[ll]\ar[dr]&&\txt{1\\2}\ar@{.>}[ll]  \\
2\ar[ur] &&*+[o][F-]{1}\ar@{.>}[ll] \ar[ur]&&\txt{1\\2}[1]\ar@{.>}[ll]\ar[ur] && 2[2]\ar@{.>}[ll] \ar[ur]&&*+[o][F-]{1[2]}\ar@{.>}[ll] \ar[ur]&&2 \ar@{.>}[ll]\ar[ur]  }$
\end{center}
{\setlength{\jot}{-3pt}
 It is easy to check that the direct sum $T=1\oplus 2[1]\oplus \begin{aligned}1\\2\end{aligned}[2]\oplus 1[2]$} of the encircled indecomposable objects is a cluster tilting object. Thus it is also a $T[1]$-cluster tilting object. Clearly, $$\{M\in\C\ |\ [T[1]](T, M[1])=0\}=\C\neq \add T,$$
 which means that (\ref{d1}) or (\ref{d2}) does not hold.

\section*{Acknowlegement}

The second author would like to thank Huimin Chang for helpful discussions on the subject.


\begin{thebibliography}{99}
\bibitem[AR]{ar} M. Auslander, I. Reiten. \newblock Applications of contravariantly finite subcategories.
\newblock Adv. Math. \textbf{86}(1), 111-152, 1991.

\bibitem[AIR]{AIR}
T. Adachi, O. Iyama and I. Reiten.
\newblock $\tau$-tilting theory.
\newblock Compos. Math. \textbf{150}(3), 415-452, 2014.

\bibitem[Au1]{au1}
M.\ Auslander.
\newblock Representation theory of Artin
algebras I.
\newblock Comm.\ Algebra {\bf 1}, 177-268, 1974.

\bibitem[Au2]{au2}
M.\ Auslander.
\newblock Coherent functors.
\newblock Proc. Conf. Categorical Algebra (La Jolla, Calif., 1965), pp. 189-231. Springer, New York, 1966.

\bibitem[Be1]{Be1}
A. Beligiannis.
\newblock  Tilting theory in Abelian categories and related homotopical structures. Preprint, 2010.

\bibitem[Be2]{Be2}
A. Beligiannis. Rigid objects, triangulated subfactors and abelian localizations. Math. Z. \textbf{274 }(3-4), 841-883, 2013.



\bibitem[BBT]{BBT}
L. Beaudet, T. Br\"{u}stle and G. Todorov.
\newblock Projective dimension of modules over cluster-tilted algebras.
\newblock Algebr. Represent. Theory \textbf{17}, 1797-1807, 2014.


\bibitem[BK]{BK}
A. I. Bondal and M. M. Kapranov.
\newblock Representable functors, Serre functors, and mutations.
\newblock Math. USSR-Izv. \textbf{35}(3), 519-541, 1990.

\bibitem[BMR]{BMR}
A. B. Buan, R. J. Marsh and I. Reiten.
\newblock Cluster-tilted algebras.
\newblock Trans. Amer. Math. Soc. \textbf{359}(1), 323-332, 2007.

\bibitem[BMRRT]{bmrrt} A. Buan, R. Marsh, M. Reineke, I. Reiten, G. Todorov. \newblock Tilting theory and cluster combinatorics.  \newblock Adv. Math. \textbf{204}(2), 572-618, 2006.

\bibitem[CZZ]{CZZ}
W. Chang, J. Zhang and B. Zhu.
\newblock On support $\tau$-tilting modules over endomorphism algebras of rigid objects.
\newblock Acta Math. Sin. (Engl. Ser.) \textbf{31}(9), 1508-1516, 2015.

\bibitem[FL]{FL}
C. Fu and P. Liu.
\newblock Lifting to cluster-tilting objects in 2-Calabi-Yau triangulated categories.
\newblock Comm. Algebra \textbf{37}(7), 2410-2418, 2009.

\bibitem[Ha]{Ha}
D. Happel.
\newblock Triangulated categories in the representation theory of finite-dimensional algebras.
\newblock London Math. Soc., Lecture Note Ser., \textbf{119}, Cambridge Univ. Press, Cambridge, 1988.


\bibitem[HJ]{hj}  T. Holm, P. J{\o}rgensen.  \newblock On a cluster category of infinite Dynkin type, and the relation to triangulations of the infinity-gon. \newblock Math. Z.  \textbf{270}(1-2), 277-295, 2012.


\bibitem[IJY]{ijy}
\newblock O.\ Iyama, P.\ J{\o}rgensen, D.\ Yang.
 \newblock Intermediate co-$t$-structures, two-term silting objects, $\tau$-tilting modules, and torsion classes. Algebra \& Number Theory, \textbf{8}(10), 2413--2431, 2014.


\bibitem[IY]{IY}
O.\ Iyama, Y.\ Yoshino.
\newblock Mutation in triangulated categories and rigid Cohen-Macaulay modules.
\newblock Invent. Math.  \textbf{172}(1), 117-168, 2008.

\bibitem[Ja]{ja} G. Jasso. \newblock  Reduction of $\tau$-tilting modules and torsion pairs.
\newblock Int. Math. Res. Not. \textbf{16}, 7190-7237, 2015.


\bibitem[KR]{kr}
B. Keller, I. Reiten.
\newblock Cluster-tilted algebras are Gorenstein and stably Calabi-Yau.
\newblock Adv. Math. \textbf{211}, 123-151, 2007.

\bibitem[KZ]{KZ}
S. Koenig, B. Zhu.
\newblock From triangulated categories to abelian categories: cluster tilting in a general framework.
\newblock Math. Z. \textbf{258}, 143-160, 2008.

\bibitem[La]{la}
A. Lasnier.
\newblock Projective dimensions in cluster-tilted categories. arXiv: 1111.3077, 2011.

\bibitem[Ne]{ne} A.\ Neeman. \newblock Triangulated Categories. Vol. 148. Princeton University Press, 2001.

\bibitem[Ng]{Ng} P. Ng. \newblock A characterization of torsion theories in the cluster category of Dynkin type $\mathbb{A}_\infty$. arXiv:1005.4364, 2010.

\bibitem[RVdB]{RVdB}
I. Reiten and M. Van den Bergh.
\newblock Noetherian hereditary abelian categories satisfying Serre duality.
\newblock J. Amer. Math. Soc. \textbf{15}(2), 295-366, 2002.

\bibitem[Sm]{Sm}
D. Smith.
\newblock On tilting modules over cluster-tilted algebras.
\newblock Illinois J. Math. \textbf{52}(4), 1223-1247, 2008.

\bibitem[YZ]{YZ}
W. Yang, B. Zhu. \newblock Relative cluster tilting objects in triangulated categories. arXiv:1504.00093, 2015, to appear in Trans. Amer. Math. Soc.

\bibitem[YZZ]{YZZ}
W. Yang, J. Zhang and B. Zhu.
\newblock On cluster-tilting objects in a triangulated category with Serre duality,
\newblock Comm. Algebra \textbf{45}(1), 299-311, 2017.

\bibitem[Zh]{Zh}
B. Zhu.
\newblock Cluster-tilted algebras and their intermediate coverings.
\newblock Comm. Algebra \textbf{39}, 2437-2448, 2011.

\bibitem[ZZ]{zz} Y.\ Zhou, B.\ Zhu. \newblock Maximal rigid subcategories in $2$-Calabi-Yau triangulated categories.
\newblock J. Algebra, \textbf{348}: 49-60, 2011.
\end{thebibliography}
\end{document}